\documentclass[11pt]{article}

\usepackage{amssymb, amsmath, mathtools, amsfonts, times, amsthm, amscd, setspace, leftidx, enumerate, xfrac}

\input xy 
\xyoption{all}

\usepackage{tikz-cd}

\usepackage{pdflscape}
\usepackage{adjustbox}

\usepackage[titletoc,title]{appendix}

\usepackage{graphicx}
\graphicspath{ {Graphics/} }
\usepackage{capt-of}

\theoremstyle{plain}

\newtheorem{proposition}{Proposition}[section]
\newtheorem{defi}[proposition]{Definition}
\newtheorem{lemma}[proposition]{Lemma}
\newtheorem{thm}[proposition]{Theorem}
\newtheorem{corollary}[proposition]{Corollary}
\newtheorem{rmk}[proposition]{Remark}

\newtheorem{ex}[proposition]{Example}

\newcommand{\tn}[0]{\otimes}

\newcommand{\ct}[1]{\mathcal{#1}}
\newcommand{\ov}[1]{\overline{#1}}

\newcommand{\un}[0]{\mathtt{1}}

\newcommand{\und}[1]{\underline{#1}}

\newcommand{\SL}[0]{\mathrm{SupLat}}

\newcommand{\op}[0]{\mathrm{op}}

\newcommand{\ev}[0]{\mathrm{ev}}

\newcommand{\cv}[0]{\mathrm{coev}}

\newcommand{\id}[0]{\mathrm{id}}

\newcommand{\Rr}[0]{\mathcal{R}}
\newcommand{\lb}[2]{\lbrace (#1,#2) \rbrace}
\newcommand{\Ll}[0]{\mathcal{L}}

\newcommand{\st}[1]{\lbrace #1 \rbrace}
\newcommand{\Pl}[1]{\mathcal{P}(#1)}

\allowdisplaybreaks
\begin{document}

\title{Skew Braces as Remnants of Co-quasitriangular Hopf Algebras in $\SL$}
\author{Aryan Ghobadi \\ \small{Queen Mary University of London }\\\small{ School of Mathematics, Mile End Road}\\\small{ London E1 4NS, UK }\\ \small{Email: a.ghobadi@qmul.ac.uk}}
\date{}

\maketitle
\begin{abstract}
Skew braces have recently attracted attention as a method to study set-theoretical solutions of the Yang-Baxter equation. Here, we present a new approach to these solutions by studying Hopf algebras in the category, $\SL$, of complete lattices and join-preserving morphisms. We connect the two methods by showing that any Hopf algebra, $\ct{H}$ in $\SL$, has a corresponding group, $R(\ct{H})$, which we call its remnant and a co-quasitriangular structure on $\ct{H}$ induces a YBE solution on $R(\ct{H})$, which is compatible with its group structure. Conversely, any group with a compatible YBE solution can be realised in this way. Additionally, it is well-known that any such group has an induced secondary group structure, making it a skew left brace. By realising the group as the remnant of a co-quasitriangular Hopf algebra, $\ct{H}$, this secondary group structure appears as the projection of the transmutation of $\ct{H}$. Finally, for any YBE solution, we obtain a FRT-type Hopf algebra in $\SL$, whose remnant recovers the universal skew brace of the solution.
\end{abstract}

\begin{footnotesize}2020\textit{ Mathematics Subject Classification}: 16T25, 18M15, 16T99, 17B37
 \end{footnotesize}
\section{Introduction}\label{SInt}
Originally appearing in statistical mechanics \cite{yang1967some}, the Yang-Baxter equation and its solutions play a fundamental role in the theory of quantum groups, braided categories and knot theory. One of the simplest realisations of this equation is over sets: we call a set $X$ and a map $r:X\times X\rightarrow X \times X$, a \emph{set-theoretical solution to the Yang-Baxter equation} (YBE) if 
\begin{equation}\label{EYBE}
r_{12} r_{13} r_{23}= r_{23} r_{13} r_{12}
\end{equation}
holds, where $r_{ij}: X\times  X\times X\rightarrow X\times X\times X$ are the applications of $r$ to the $i$ and $j$-th components of $X^{3}$ i.e. $r_{23}= \mathrm{id}\times r$.  
In \cite{drinfeld1992some}, Drinfeld proposed the classification of such solutions as an open problem. Thenceforth, these objects have garnered large interest due their interactions with combinatorics \cite{gateva2004combinatorial}, ring theory \cite{smoktunowicz2018skew,smoktunowicz2018engel} and their applications to knot theory \cite{carter2004homology}. More recently, the work of Rump on involutive YBE solutions \cite{rump2007braces,rump2006modules} inspired Guarnieri and Vendramin to develop of the theory of \emph{skew braces}, which are sets with two compatible group structures \cite{guarnieri2017skew}. In particular, any set-theoretical YBE solution has a corresponding universal skew brace, allowing us to classify set-theoretical YBE solutions, by first classifying such algebraic structures. However, when looking at linear YBE solutions on vector spaces, there is a well established correspondence between these solutions and (co-)quasitriangular Hopf algebras. (Co-)quasitriangular Hopf algebras, provide solutions of YBE via their representation theory, and conversly the Fadeev-Reshitkhin-Takhtajan (FRT) construction produces such a Hopf algebra, from a given YBE solution. From the latter, we see that skew braces replace Hopf algebras, in the set-theoretical world. Hence, it would be natural to ask whether skew braces are related to Hopf algebras in a suitable category related to sets. If so, this relation should allow us to \textbf{(a)} apply the usual categorical Hopf algebra techniques to obtain new skew braces, \textbf{(b)} use the FRT construction and recover the universal skew brace and \textbf{(c)} explain the nature of the two products on a skew brace and their interaction, which has been subject to several studies already. In this work, we show that the correct category to consider is that of complete lattices and join preserving morphisms, $\SL$, and construct skew braces from coquaitriangular Hopf algebras in this category and vice-versa.

The Hopf algebra point of view fails when studying set-theoretical YBE solutions because of two key reasons: 
\begin{enumerate}[(A)]
\item Hopf algebras in the category of sets and functions, $\mathrm{Set}$, are groups and it is easy to check that any (co)quasitriangular structure on a group must be trivial. Hence, we can not obtain YBE solutions by looking at (co)modules over a group in $\mathrm{Set}$.
\item The key ingredient to the FRT construction is Tannaka-Krein duality, which requires the underlying object of the YBE solution to be dualizable, while the only dualizable object in $\mathrm{Set}$ is the set of one element. 
\end{enumerate}
The first naive solution is to look at the category of sets and relations, $\mathrm{Rel}$, where every set has itself as a dual, making the category rigid. However, $\mathrm{Rel}$ is not cocomplete and the colimit needed for Tannaka-Krein reconstruction, \ref{Ecoend}, will not exist. The second naive solution is to move into the cocompletion of $\mathrm{Rel}$, namely $[\mathrm{Rel}^{\op},\mathrm{Set}]$, via the Yoneda embedding. But this category is rather large and the Hopf algebras constructed will not be very intuitive. Instead, we remedy these issues by embedding the category of sets into the category of complete lattices and join-preserving morphisms, $\SL$, via the power-set functor: 
$$\xymatrix@C+48pt@R-17pt{ (\mathrm{Set},\times ,1)  \ar[r]^{\mathrm{inc.}}_{\tiny\mathrm{strong\  monoidal}} & (\mathrm{Rel}, \times ,1) \ar[r]^-{\Pl{-}}_-{\tiny\mathrm{strong\  monoidal}} & (\SL
,\tn,\Pl{1})\\ \small\txt{YBE\  solutions \\\ }\ar@{<->}[r] & \small\txt{YBE\  solutions \\\  } \ar@{<->}[r]& \small\txt{YBE solutions \\ on rigid objects}}$$
In particular, all objects of the form $\Pl{X}$, for a set $X$, are dualizable and conversly YBE solutions on rigid objects of $\SL$, provide set-theoretical YBE solutions. The other major benefit of working in $\SL$, is that we can formulate a deeper connection between co-quasitriangular Hopf algebras in this category and groups with braiding operators, which are groups with a compatible YBE solutions on their underlying sets, see \ref{EbrdOpr1} and \ref{EbrdOpr2}. 

Our results can be summarised as follows: Given a Hopf algebra structure on a complete lattice $\ct{H}$ in $\SL$, we can form a new Hopf algebra by quotienting out the ``kernel" of the counit, Lemma \ref{LQuotHpf}. The counit of this new Hopf algebra will send all non-trivial elements to $1\in \ct{P}(1)$ and in Lemma \ref{LHpfGrp}, we show that this condition is equivalent to the Hopf algebra being the ``group algebra", see Example \ref{EHpfGrpAlg}, of a group. Hence, this process provides a corresponding group for every Hopf algebra in $\SL$, which we call its \emph{remnant} and denote by $R(\ct{H})$. 

It is well-known that the multiplication of a co-quasitriangular Hopf algebra is braided-commutative with respect to a naturally induced braiding, \ref{EBrdCQHA}, on the Hopf algebra. Hence, we demonstrate that given a co-quasitriangular structure on $\ct{H}$, the induced braiding of $\ct{H}$ restricts to a braiding operator on its remnant, Theorem \ref{TRemCQHA}. Additionally, any group with a braiding operator, possesses a secondary group structure on the same set, which makes it a skew brace. On the other hand, Majid has shown that any co-quasitriangular Hopf algebra has an induced secondary multiplication and an antipode which provide it with a braided Hopf algebra structure, called its \emph{transmutation}, in its category of comodules \cite{majid1993transmutation}. A corollary of our work is that the secondary group structure on the remnant agrees with the projection of the transmuted multiplication of $\ct{H} $, Theorem \ref{TRemSkwB}. Finally, in Section \ref{SSkwCQHA}, we show that any skew brace can be recovered as the remnant of a co-quasitriangular Hopf alebra in $\SL$.

We must point out that similar ideas were discussed in \cite{LYZ1,LYZ2}, where Hopf algebras in $\mathrm{Rel}$ are shown to correspond to groups with unique factorisations, $G=G_{+}G_{-}$, and quasitriangular structures on them are fully described. Although the theory is presented for finite dimensional \emph{positive} Hopf algebras, the authors of \cite{LYZ1} are aware that the proofs should work for any finite free $\mathbb{B}$-module, where $\mathbb{B}$ is the Boolean algebra with two elements. Their work translates into the classification of Hopf algebra structures on free lattices i.e. lattices which are of the form $\Pl{X}$ for a set $X$, in $\SL$ and the finiteness condition can be completely avoided due to the rigidity of $\Pl{X}$. We review these results and briefly comment on their proofs in Section \ref{SFrHpfSL}. In \cite{LYZ}, the authors describe the properties of the universal group of a set-theoretical solution by taking inspiration from \cite{LYZ1,LYZ2}, but do not directly connect the works. By providing the correct categorical setting i.e. $\SL$ (which goes beyond the category of $\mathbb{B}$-modules), we are able to present a single machinery which captures both constructions, namely by viewing them as remnants of co-quasitriangular Hopf alebras in $\SL$. 

It has been observed that different set-theoretical YBE solutions can have isomorphic universal skew braces. However, when applying the FRT construction in Section \ref{SFRT}, we shall see that the Hopf algebra associated to the solution, remembers a large part of the original solution. It is only when we take the remnant of the Hopf algebra, that much of this additional data is lost. Naturally, non-isomorphic co-quasitriangular Hopf alebras can have isomorphic remnants and should provide stronger invariants of set-theoretical YBE solutions, while being more difficult to work with. The additional benefit of working in this setting is that one can utilise classical Hopf algebraic techniques such as (co-)double bosonasation to produce new examples of skew braces, however, this will be discussed in another place.

In Appendix \ref{App}, we discuss a natural bijection $\mathfrak{l}:X\rightarrow X$, which is induced when the set $X$ is equipped with a YBE solution $r$. This bijection comes into play when we view $(X,r)$ as a dualizable object in $\mathrm{Rel}$ and appears again in the our reconstructed Hopf algebra for the solution. 

\textbf{Acknowledgements.} The author would like to thank Shahn Majid, for many helpful discussions on the topic 

\section{Preliminaries}\label{SPre}

Throughout the article and particularly, in Sections \ref{SBasics} and \ref{SFRT}, we assume that the reader is familiar with the notion of Hopf algebras in symmetric (braided) monoidal categories and refer to Chapter 9 of \cite{majid2000foundations} for any details which we have left out. We will however present the structure of Hopf algebras in $\SL$ and co-quasitriangular structures on them as definitions, in Section \ref{SHpfSL}, and previous knowledge of Hopf algebras is not essential for the rest of the article.  

\textbf{Notation.} All Hopf algebras considered in this work, will have invertible antipodes, and as noted later, all YBE solutions considered are assumed to be non-degenerate. We will freely use either $m$ and $.$ to denote the multiplication operation, unless otherwise stated. We will denote the elements of the quotient of a set $S$, by $\ov{a}\in S/\sim $, for $a\in S$, unless this is clear from context, in this case we will simply write $a$. All monoidal categories are assumed to be strict, since the associator and other structural morphisms will be trivial in the concerned example. If not stated otherwise a lower case letter such as $x$, will be an element of the set denoted by the upper case lettering, $X$. The number $1$ will both denote the unit element in our constructions and in the right context will denote the set with one element and $\Pl{1}=\st{\emptyset, 1}$. For an arbitrary monoidal category $(\ct{C},\tn ,\un)$, we say an object $X$ in $\ct{C}$ has a. right dual $X^{\vee}$, with duality morphisms $\ev_{X} : X\tn X^{\vee} \rightarrow \un $ and $\cv_{X}:\un \rightarrow X^{\vee} \tn X$ if $(\ev_{X}\tn \id_{X})(\id_{X}\tn \cv_{X})=\id_{X}$ and $(\ev_{X}\tn \id _{X^{\vee}})(\id _{X^{\vee}}\tn \cv_{X})=\id_{X^{\vee}}$ and $X^{\vee}$ will always denote the right dual of $X$.

\subsection{FRT Reconstruction}\label{SBasics}
In this section, we review some necessary results about FRT reconstruction. These results are also present in \cite{majid2000foundations} and go back to \cite{majid1993braided}. However, our presentation of the results directly in terms of the coend is closer to \cite{shimizu2019tannaka}. In particular, we adapt the notation used in \cite{shimizu2019tannaka}.

Let $\tilde{\ct{B}}$ be denote the \emph{rigid extension of the category of braids}, which is the ``smallest" rigid braided monoidal category. Explicitly, $\tilde{\ct{B}}$ is the monoidal category generated by two objects, $\mathbf{x}$ and $\mathbf{y}$, $\ct{B}$ and morphisms $\ev : \mathbf{x}\tn\mathbf{y} \rightarrow\un$, $\cv :\un  \rightarrow \mathbf{y}\tn \mathbf{x}$ and an invertible morphisms $\kappa_{\mathbf{a}, \mathbf{b}} :\mathbf{a}\tn \mathbf{b} \rightarrow \mathbf{b} \tn \mathbf{a} $ for $\mathbf{a},\mathbf{b}\in\lbrace \mathbf{x}, \mathbf{y}\rbrace$ with the relevant relations [Definition 6.12 \cite{shimizu2019tannaka}], which make $\kappa$ the braiding of the category and $\mathbf{y} $ the right dual of $\mathbf{x}$. 

A pair $(X, r)$ is called a \emph{braided object} or \emph{YBE solution}, \cite{joyal1993braided}, in a monoidal category $(\ct{C},\tn, \un )$, if $X$ is an object of $\ct{C}$ and $r:X\tn X\rightarrow X\tn X$ an invertible morphism satisfying 
\begin{equation}\label{ECatYBE}
(\id_{X}\tn r) (r\tn\id_{X})(\id_{X}\tn r)= (r\tn\id_{X})(\id_{X}\tn r)	(r\tn\id_{X})
\end{equation}
If $X$ has a right dual $X^{\vee}$, with duality morphisms $\ev_{X} : X\tn X^{\vee} \rightarrow \un $ and $\cv_{X}:\un \rightarrow X^{\vee} \tn X$, we say the braided object $(X,r)$ is \emph{dualizable} if the morphisms 
\begin{align}
r^{\flat}:= &( \id_{X^{\vee}\tn X}\tn\ev)( \id_{X^\vee} \tn r \tn \id_{X^{\vee}})(\cv \tn \id_{X\tn X^{\vee}})
\\ (r^{-1})^{\flat}:= &( \id_{X^{\vee} \tn X}\tn\ev)( \id_{X^\vee} \tn r^{-1} \tn \id_{X^{\vee}})(\cv \tn \id_{X\tn X^{\vee}})
\end{align}
are invertible. 

Given a dualizable YBE solutions, one can define a functor $\omega :\tilde{\ct{B}}\rightarrow \ct{C} $ by $\omega(\mathbf{x})= X$, $\omega(\mathbf{y})= X^{\vee}$ and 
\begin{align*}
\omega ( \kappa_{\mathbf{x}, \mathbf{x}})= r,\ \omega ( \kappa_{\mathbf{x}, \mathbf{y}})&= (r^{-1})^{\flat},\ \omega ( \kappa_{\mathbf{y}, \mathbf{x}})= (r^{\flat})^{-1},\ \omega ( \kappa_{\mathbf{y}, \mathbf{y}})= r^{\vee} 
\\\omega &(\cv)=\cv_{X},\ \omega (\ev)=\ev_{X}
\end{align*}

Given a YBE solution in a category, it is natural to ask which other braided objects are generated by this solution. For example, if $(X,r) $ is a braided object, then tensor products of $X$ also have induced braidings e.g. $X\tn X $ with the induced braiding $ (\id_{X}\tn r\tn \id_{X})(r\tn r)(\id_{X}\tn r\tn \id_{X})$. The direct approach is to generate the largest category of braided objects generated by $(X,r)$. Explicitly, one starts with $(X,r)$, and adds objects by performing tensor products and possible products, coproducts, etc. Finally, in the spirit of Lyubashenko's work \cite{lyubashenko1986hopf}, one can hope to realise this category as comodules over a Hopf algebra. 

The second point of view, is that objects generated by $(X,r)$ will be those, which ``braid past, what  $(X,r)$ braids past". To make this statement more explicit, we need to recall the definition of the \emph{dual of monoidal functors} or \emph{weak centralizer}. The \emph{dual of a strong monoidal functor} $U:\ct{D}\rightarrow \ct{C}$ is defined as the category whose objects are pairs $(X,\sigma : X\tn U \Rightarrow U\tn X )$, where $X$ is an object of $\ct{C}$ and $\sigma$ a natural (monoidal) isomorphism satisfying $\sigma_{\mathtt{1}}=\id_{X}$ and 
$$(\id_{M}\otimes \tau_{N})(\tau_{M}\otimes \id_{N})=(U_{2}(M,N)^{-1}\otimes \id_{X})\tau_{M\otimes N}(\id_{X}\otimes U_{2}(M,N))$$
where $M,N$ are objects of $\ct{C}$. We denote the dual by $\ct{W}(U)$ and the \emph{lax} left dual, where $\tau$ is not assumed to be isomorphisms by $\ct{W}_{l}(U)$. Note that the left (lax) dual lifts the monoidal structure via  
$$ (X,\tau )\otimes (Y, \rho):= (X\otimes Y , (\tau \otimes \id_{Y})(\id_{X}\otimes\rho ) )$$ 
with $(\mathtt{1}_{\mathcal{C}}, \id_{X})$ as its unit and the strict monoidal functor $\overline{U} :\ct{W}_{(l)}(U)\rightarrow \mathcal{C}$ defined by $\overline{U} (X,\tau )=X$.

With the notion of the dual in mind, given a braided object and its corresponding functor $\omega :\tilde{\ct{B}}\rightarrow \ct{C} $, all the braided objects generated by $(X,r)$ must braid past objects of $\ct{W}(U)$. In particular if $\ov{U}: \ct{W}(U)\rightarrow \ct{C}$ denotes the corresponding forgetful functor, then all the generated braided object will be objects of $\ct{W}_{l}(\ov{U})$. We now recall, the conditions under which this category can be recovered as the comodule category of a co-quasitriangular Hopf algebra \cite{majid1993braided}. 

Let $\ct{D}$ be a small monoidal category and $\omega: \ct{D}\rightarrow \ct{C}$ be a strict monoidal functor. We consider the functor  $\omega\tn\omega^{\vee}:\ct{D}\times \ct{D}^{\op}\rightarrow \ct{C}$ and denote its \emph{coend} [Chapter IX.6 \cite{mac2013categories}], if it exists, by 
\begin{equation}\label{Ecoend}
H_{\omega}:=\int^{a\in \ct{D}}	\omega(a)\tn\omega(a)^{\vee}
\end{equation}
Recall that the coend is the colimit of the diagram consisting of objects $ \omega(a)\tn\omega(b)^{\vee} $ and parallel pairs $\omega (f)\tn \id_{\omega(b)^{\vee}}, \id_{\omega (a)}\tn \omega(f)^{\vee} $ corresponding to objects $a,b\in \ct{D}$ and morphisms $f:a\rightarrow b$ in $\ct{D}$, respectively. 

\begin{thm}\label{TFRTmain} If $(\ct{C},\Psi)$ is a symmetric monoidal category, and the mentioned coend exists, it comes equipped with the structure of a bialgebra, such that $\ct{W}_{l}(\ov{U})$ is monoidal equivalent to the the category of left $H_{\omega}$-comodules, $\prescript{H_{\omega}}{}{\ct{C}}$. Additionally, if $\ct{D}$ is rigid, then $H_{\omega} $ admits a bijective antipode, making it a Hopf algebra object in $\ct{C}$. If $\ct{D}$ is braided, then $H_{\omega}$ has an induced co-quasitriangular structure.  
\end{thm}

Here we only recall the induced co-quasitriangular Hopf alebra structure on $H_{\omega}$. The proof of this result can be found in Chapter 9 of \cite{majid2000foundations}, where the coend is is described in terms of natural transformations between certain functors or in Theorem 4.3 of \cite{shimizu2019tannaka}, which uses the language we will present it as.

Because of the simplicity of our examples, we will be assuming that the functor $\omega$ is strict monoidal and additionally $\omega(x)^{\vee}=\omega (x^{\vee})$ for all $x\in \ct{D}$. Let $\mu_{x}: \omega(x)\tn\omega(x)^{\vee}\rightarrow H_{\omega} $ denote the unique natural morphisms, making $H_{\omega}$ the colimit of the diagram. The Hopf algebra structure on $H_{\omega} $ consists of $(m,\eta,\Delta, \epsilon, S)$ which are the unique morphisms satisfying:
\begin{align*}
m:H_{\omega}\tn H_{\omega}\rightarrow H_{\omega}\ ;&\quad m(\mu_{x}\tn \mu_{y})=\mu_{x\tn y}\big( \id_{\omega(x)}\tn \Psi_{\omega(x)^{\vee},\omega(y)}\tn \id_{\omega(y)^{\vee}}\big) 	
\\  \eta :\un \rightarrow H_{\omega}\ ;& \quad \eta=\mu_{\un}
\\\Delta: H_{\omega}\rightarrow H_{\omega}\tn H_{\omega}\ ;& \quad\Delta\mu_{x} = (\mu_{x}\tn\mu_{x})\big(\id_{\omega(x)}\tn \cv_{\omega(x)}\tn \id_{\omega(x)^{\vee}}\big)
\\\epsilon :H_{\omega}\rightarrow \un\ ;& \quad \epsilon\mu_{x}=\ev_{\omega(x)}
\\S: H_{\omega}\rightarrow H_{\omega}\ ;\quad S\mu_{x}&= \mu_{x^{\vee}}(\ev_{\omega(x)}\tn\id_{\omega(x^{\vee})\tn \omega(x^{\vee\vee})})( \Psi_{\omega(x^{\vee})\tn \omega(x^{\vee\vee}),\omega(x^{\vee}) })
\\ &\quad ( \id_{\omega(x)\tn \omega(x)^{\vee}}\tn \omega(\ev_{x})^{\vee})
\end{align*}
where $x,y\in \ct{D}$. As we will see in our example, $\Psi$ and $\ev,\cv$ are rather trivial and the expressions will be much simpler to deal with. The key ingredient which we need, lies in the co-quasitriangular structure $\Rr:H_{\omega}\tn H_{\omega} \rightarrow \un$ induced on $H_{\omega}$, when $\ct{D}$ is a braided category. If $\psi$ denotes the braiding of $\ct{D}$, then $\Rr $ is the unique morphism satisfying
\begin{equation*}
\Rr (\mu_{x}\tn\mu_{y})=\ev_{\omega(y\tn x)} (\omega (\psi_{x,y})\tn \id_{\omega(x)^{\vee}\tn\omega(y)^{\vee}})( \id_{\omega(x)}\Psi_{\omega (x)^{\vee},\omega (y)}\tn \id_{\omega(y)^{\vee}})	
\end{equation*}
for any pair of objects $x,y\in \ct{D}$.   
\subsection{Set-theoretical YBE Solutions and Skew Braces}\label{SSkwBasic}
In this section, we review the theory of Skew braces, with \cite{smoktunowicz2018skew} as our main reference. 

We have already defined what a set-theoretical YBE solution $(X,r)$ is in the introduction. We will use the notation $r(x,y)= (\sigma_{x}(y) ,\gamma_{y}( x) )$, for maps $\sigma_{x},\gamma_{y}:X\rightarrow X$ corresponding to elements $x,y\in X$. The YBE solution is said to be \emph{non-degenerate} if $\sigma_{x},\gamma_{y}$ are bijections for all $x,y\in X$. For such a solution, where $r$ is bijective, we adapt the notation $r^{-1}(x,y)= (\tau_{x}(y) ,\rho_{y}( x) )$, for maps $\tau_{x},\rho_{y}:X\rightarrow X$ corresponding to elements $x,y\in X$. The solution is called \emph{involutive} if $\sigma_{x}=\tau_{x}$ and $\gamma_{x}=\rho_{x}$ for all $x\in X$.

\textbf{Notation.} We will only consider non-degenerate solutions and from here forward a set-theoretical YBE solution will refer to a non-degenerate one.

We now recall the definition of the a braiding operator on a group and the definition of the universal group of a set-theoretical YBE solution from \cite{LYZ}.

A braiding operator on a group $(G,m,e)$ is a map $r:G\times G\rightarrow G\times G$ satisfying 
\begin{align}
 mr(a,b)&= a.b
\\ r(e,g)=(g,e), \quad  &r(g,e) =(e,g) 
\\r(g.h, f) = (\id_{G}\times m)&(r\times \id_{G})(g, r(h,f))\label{EbrdOpr1}
\\r(g,h.f)=(m\times \id_{G}) &(\id_{G}\times r)(r(g,h),f) \label{EbrdOpr2}
\end{align}
for any $g,h,f\in G$. It follows from these axioms, that $r$ is invertible and satisfies the YBE equation [Corollary 1 of \cite{LYZ}] i.e. $(G,r)$ is a braided object in $\mathrm{Set}$. A pair $(G,r)$ is sometimes referred to as a \emph{braided group}. We avoid this term, since in our main reference \cite{majid2000foundations}, the term ``braided group" is used to discuss braided Hopf algebras in braided monoidal categories. 

Given a set-theoretical YBE solution $(X,r)$, we consider the group 
$G(X,r)= F _{g}(X)/ \langle x.y= \sigma_{x}(y).\gamma_{y}(x) \mid x,y\in X\rangle$,
where $F_{g}(X)$ denotes the free group generated by the set $X$. Observe that the braiding on $X$, extends to a braiding operator $\ov{r}$ on $G(X,r)$, defined by $\ov{r}(\ov{x},\ov{y})= \big(\ov{\sigma_{x}(y)}, \ov{\gamma_{y}(x)} \big)$. Additionally, the natural map $i: X\rightarrow G(X,r)$ defined by $x\mapsto \ov{x}$ commutes with the braiding operators on both sets. One must keep in mind that $i $ is not necessarily injective, but satisfies a universal property with respect to groups with braiding operators. Explicitly, if $(H,s)$ is a group with a braiding operator and $f:X\rightarrow H$, a map which commutes with the respective braidings, then $f$ must factorise through $i$. 

Now we review the theory of Skew braces with reference to \cite{smoktunowicz2018skew}, although we adapt the notation of \cite{bachiller2018solutions,bachiller2016study}, which is compatible with \cite{LYZ}. 

A \emph{skew left brace} consists of a set $B$ with two group structures $(B,.)$ and $(B,\star )$, satisfying 
\begin{equation}\label{ESkwBr}
a.(b\star c) = (a.b)\star a^{\star} \star (a.c)	\end{equation}
for all $a,b,c\in B$, where we denote the multiplicative inverse of $a$ with respect to $.$ and $\star$ by $a^{-1}$ and $a^{\star}$, respectively. A skew \emph{right} brace can be defined accordingly [Definition 2.1 \cite{bachiller2018solutions}]. However, as we will see in Remark \ref{RRightBr}, the theory of right skew braces is symmetric to that of skew left braces. Hence, in this paper we will only work with skew left braces and refer to them as \emph{skew braces}. 

The reader should be careful when referring to other works related to skew braces, since many other sources use the notation $\circ$ and $.$ instead of $.$ and $\star$, respectively. 

Skew braces are in fact equivalent to groups with braiding operators. This is in fact proved in Theorem 2 of \cite{LYZ}, but is stated in the language of skew braces in Section 3 of \cite{smoktunowicz2018skew}. 
\begin{thm}\label{TSkwMain} Given a group $(G,.)$, the following additional structures on $G$ are equivalent:
\begin{enumerate}[(I)]
	\item A second group structure $(G,\star)$, which makes $G$ a skew brace.
	\item A braiding operator $r:G\times G\rightarrow G\times G$ on the group $G$.
\end{enumerate}
\end{thm}
We briefly recall this correspondence from the mentioned sources. Given a skew brace structure as in part (I), the induced braiding operator on $G$ is defined by 
\begin{equation}\label{EBrdofSkw}
(a,b)\longmapsto \big( a^{\star}\star(a.b),(a^{\star}\star(a.b))^{-1}.a.b\big)\end{equation}
for $a,b\in G$. Conversely, given a braiding operator $r$ with the usual notation $r(x,y)=(\sigma_{x}(y),\gamma_{y}(x))$, the operation $x\star y:= x.\sigma_{x}^{-1}(y)=x.\sigma_{x^{-1}}(y)$ defines a compatible group structure on $G$. In fact there's a third equivalent structure namely the existence of a bijective 1-cocycle, which we will not discuss. 

For a skew brace $(B,.,\star)$, by axiom \ref{EbrdOpr1}, the map $\lambda : (B,.)\rightarrow \mathrm{Aut}(B,\star)$ defined which sends an element $a\in B$ to the map $\lambda_{a}(b)=   a^{\star}\star(a.b) $, is a group morphisms. Hence, to any skew brace $(B,.,\star)$, one can associate a group structure on the set $B\times B$, namely the cross product $(B,\star) \rtimes (B,.)$, with the multiplication 
\begin{equation}\label{ECrssGrp}
(a,b).(c,d)= (a\star \lambda_{b}(c), b.d)
\end{equation}
for $a,b,c,d\in B$. This group is called the \emph{crossed group} of $(B,.,\star)$ [Definition 1.10 \cite{smoktunowicz2018skew}].

\subsection{The Category $\SL$}\label{SSL}
We briefly review the monoidal closed structure of $\SL$ and its colimits, from Chapter 1 of \cite{joyal1984extension}. The reader can refer to this source for additional details on the structure of $\SL$.

The objects of $\SL$ are partially ordered sets $(\Ll,\leq )$, where any subset $S\subseteq \Ll$, has a least upper bound i.e. an element denote by $\vee S$ such that if $l\in\Ll$ satisfies $s\leq l$ for all $s\in S$, then $\vee S\leq l$ holds. In this case, $\vee S$ is  called the \emph{supremum} or \emph{join} of $S$. If $S=\lbrace a_{i}\mid i\in I\rbrace$ for some index set $I$, we use notation $\vee_{i\in I} a_{i}$ instead of $ \vee S$. The morphisms of $\SL$ are join-preserving maps between sets. 

Notice that a partially ordered set in $\SL$ also admits all infima, i.e the greatest lower bound or \emph{meet} of a subset $S\in \ct{L}$ will be $\vee \lbrace a\mid a\leq s\ ,\forall s\in S \rbrace$ and is denoted by $\wedge S$.
Hence, all objects of $\SL$ are \emph{complete lattices}, but join-preserving morphisms between them, do not necessarily preserve infima. We denote the smallest element of every complete lattice by $\emptyset$. 

\begin{ex}\label{ECmpLat} The simplest example of a complete lattice is the power-set of a set $X$, denoted by $\ct{P}(X)$, with $\vee$ and $\wedge$ being the union and intersection of subsets, respectively. This example is often called the \emph{free lattice} on set $X$. Alternatively, we can consider the set of positive integers, $\mathrm{div}(z)$, which divide a positive integer $z$. The set $\mathrm{div}(z)$ possesses a partial ordering by division, and $\vee$ and $\wedge$ are given by the lowest common multiple and the greatest common divisor, respectively. 
\end{ex}

The category $\SL$, is complete and cocomplete. Limits in $\SL$ are easy to construct and follow directly from $\mathrm{Set}$. The equaliser of to morphisms between complete lattices is exactly the equalizer of the two underlying maps between the two underlying sets and the product of  complete lattices is the product of the underlying sets with coordinate-wise order. 

The coproduct of complete lattices $\ct{L}_{i}$ for an index set $I$ with suitable cardinality, which will always be the case in our work, will again be the product of the sets $\ct{L}_{i}$, with coordinate-wise order $\prod_{i\in I}\ct{L}_{i}$. The lattice $\prod_{i\in I}\ct{L}_{i}$ is viewed as a coproduct via the inclusion morphisms $\ct{L}_{i}\rightarrow \prod_{j\in I}\ct{L}_{j}$, which for a fixed $i\in I$ send an element $l\in \ct{L}_{i}$ to $(a_{j})_{j\in I}$, where $a_{i}= l$ and $a_{j}=\emptyset$ for $j\neq i$.

The coequalizer of a parallel pair $f,g:\ct{L}\rightrightarrows \ct{N}$ between lattices is more difficult to describe. First, we recall the description provided in Proposition 3 of \cite{joyal1984extension}: Let $\ct{K} $ be the subset of elements in $k\in\ct{N}$, which satisfy the following property: $\forall l\in \ct{L}$, either $f(l)\vee g(l) \leq k $ or neither $f(l)\leq k$ nor $g(k)\leq k$ hold. The partial order of $\ct{N}$ restricts to $\ct{K}$ and the morphism $c:\ct{N}\rightarrow\ct{K}$, defined by $n \mapsto \wedge\lbrace k\in \ct{K}\mid n \leq k \ \text{in}\ \ct{N}\rbrace$, makes $\ct{K}$, the coequalizer of the parallel pair $f,g$. 

If we compose the morphism $c$ with the inclusion of $\ct{K}$ into $\ct{N}$, we obtain a \emph{map}, no longer a join-preserving morphism of lattices, $\underline{c}: \ct{N}\rightarrow \ct{N}$. It is easy to see that $\underline{c}$ is a \emph{closure operator} i.e satisfies $\underline{c} ^{2}= \underline{c} $, $a\leq\underline{c}(a) $ and $\underline{c} (a) \leq \underline{c} (b) $ for $a,b\in \ct{N}$, where $a\leq b$. Hence, we define $\ov{\ct{K}}$ to be the quotient of set $\ct{N}$ by the equivalence relation $n\sim \underline{c}(n)$ for all $n\in \ct{N}$. In other words, $\ov{\ct{K}}$ is the coequalizer of the parallel pair $\id_{\ct{N}},\underline{c}$ in $\mathrm{Set}$. It follows that $\ov{\ct{K}}$ has a partial order defined by $\ov{a}\leq \ov{b}$ iff $\underline{c}(a)\leq \underline{c}(b)$ and thereby the natural map $\ov{()}:\ct{N}\rightarrow \ov{\ct{K}} $ defined by $n\mapsto \ov{n}$ is order-preserving. Additionally, we claim that $\ov{\vee_{i\in I} n_{i}}= \vee_{i\in I} \ov{n_{i}}$, making $\ov{\ct{K}}$ a complete lattice and $\ov{()}$ a join-preserving morphism. Notice that $\underline{c} ( \vee_{i\in I} n_{i})\geq \underline{c} (n_{i})$ holds and if $\ov{n}\geq \ov{n_{i}}$ for all $i\in I$, then 
\begin{align*}
\underline{c}(n)\geq \underline{c}(n_{i})\geq n_{i}\ \forall i\in I\ \Rightarrow \underline{c}( n)\geq \vee_{i\in I} n_{i}\ \Rightarrow \underline{c}(n)= \underline{c} ^{2}(n)\geq \underline{c} (\vee_{i\in I} n_{i})	
\end{align*}
which proves our claim. It should be clear to see that $\ov{()} $ factorises through an isomorphism between $\ct{K}$ and $\ov{\ct{K}}$. Furthermore, observe that in $\mathrm{Set}$, the map $\ov{()} : \ct{N} \rightarrow \ov{\ct{K}} $ admits a section defined by $\ov{b}\mapsto \underline{c} (b)$. This map is a join-preserving morphism iff $\underline{c}$ is join preserving.  
 
Now, let us recall the monoidal structure of $\SL$. For lattices $\ct{M},\ct{N}$ and $\ct{L}$, a map $f:\ct{M}\times \ct{N}\rightarrow \ct{L}$ is called a \emph{bimorphism} if it preserves suprema component-wise i.e. $f(\vee_{i\in I}m_{i}, n) = \vee_{i\in I} f( m_{i},n)$ and $f(m,\vee_{i\in I}n_{i})=\vee_{i\in I}(m,n_{i}) $. There exists a lattice $\ct{M}\tn \ct{N}$ with the universal property that any bimorphism $f$, as above, factorises through a morphism $\ov{f}: \ct{M}\tn \ct{N}\rightarrow \ct{L}$. Explicitly, $\ct{M}\tn \ct{N}$ can be described as the quotient of the lattice $\ct{P}(\ct{M}\times \ct{N}) $ by relations $\lbrace (\vee_{i\in I}m_{i}, n)\rbrace = \cup_{i\in I} \lbrace( m_{i},n)\rbrace$ and $\lbrace(m,\vee_{i\in I}n_{i})\rbrace=\cup_{i\in I}\lbrace(m,n_{i})\rbrace $. Although, we view $\ct{M}\tn \ct{N} $ as a quotient of $\ct{P}(\ct{M}\times \ct{N}) $, we will denote its elements by $\lbrace (m,n)\rbrace $ instead of $\ov{\lbrace (m,n)\rbrace}$. Hence, $\ct{M}\tn \ct{N} $ is ``spanned", via $\cup$, by elements of the form $\lbrace (m,n)\rbrace$ and the relations induce component-wise ordering on these elements i.e. if $m\leq m'$, then $\lbrace(m,n) \rbrace\leq \lbrace(m',n) \rbrace $ since $m\vee m'=m'$.

In fact $\tn $, defines a bifunctor and provides a symmetric monoidal structure on $\SL$. The unit of the monoidal structure is given by $\ct{P}(1)$, where $1$ denotes the set with one element. More generally, for a pair of sets $X,Y$, $\ct{P}(X)\tn \ct{L} \cong\prod_{x\in X} \ct{L}$ and $\ct{P}(X)\tn \ct{P}(X)\cong \ct{P}(X\times Y)$ [Proposition 2 \cite{joyal1984extension}]. Hence, the natural power-set functor $\ct{P}:\mathrm{Set}\rightarrow \SL$ is strong monoidal. The symmetric structure of course follows from the symmetry of $\times$ and viewing $\ct{M}\tn \ct{N}$ as a quotient of $\ct{P}(\ct{M}\times \ct{N}) $.

There exists a natural partial ordering on the hom-sets between complete lattices defined by $f\leq g$ iff $f(m)\leq g(m),\ \forall m \in \ct{M}$, for $f,g\in \mathrm{Hom}(\ct{M},\ct{N})$. By taking point-wise suprema, one can conclude that $\mathrm{Hom}(\ct{M},\ct{N})$ is also a complete lattice. Furthermore, one can prove the familiar hom-tensor adjunction, which makes $\SL$ a closed monoidal category. Next we characterise the \emph{rigid} or \emph{dualizable} objects in $\SL$.

\begin{defi}\label{DBasis} We say a lattice $\ct{L}$ has a \emph{basis} if there exists a set of elements $\lbrace l_{i} \rbrace_{i\in I}$, such that 
\begin{enumerate}[(A)]
\item For any element $ l\in \ct{L}$, there exists a (not necessarily unique) subset $J_{l}\subset I$, such that $l=\vee_{i\in J_{l}}l_{i}$
\item For any pair $i,j\in I$, $l_{j}\leq l_{i} $ iff $i=j$	
\end{enumerate}
hold. 	
\end{defi}

Observe that condition (B), implies that there are no elements under the basis elements i.e. if $\emptyset\neq l\leq l_{i}$, then $l=l_{i}$. Any free lattice $\ct{P}(X)$ admits a basis, consisting of singleton subsets $\lbrace x\rbrace $ for elements $x\in X$. Additionally, $\mathrm{div}(6)$ admits a basis $\lbrace 2, 3\rbrace$, while $\mathrm{div}(4)$ does not. More generally, $\mathrm{div}(z)$ admits a basis iff the power of all primes present in the prime factorisation of $z$, is one.  

It's easy to see that any complete lattice $\ct{L}$ with basis $\lbrace l_{i}\rbrace_{i\in I}$, is a rigid object in $\SL$, with itself as its dual. Since $\ct{L}$ has a basis, then $\ct{L}\tn\ct{L}$ also admits a basis with elements $\lb{l_{i}}{l_{j}}$ for pairs $i,j\in I$. Hence, it is straightforward to check that morphisms $\cv : \ct{P}(1)\rightarrow \ct{M}\tn \ct{L} $ defined by $\cv(1)= \cup_{i\in I}\lb{l_{i}}{l_{i}}$ and $\ev : \ct{L}\tn \ct{L}\rightarrow \ct{P}(1)$ defined by $\ev \lb{l_{i}}{l_{j}}=1$ iff $i=j$ which extends to $\Ll\tn\Ll$, by a choice of $J_{l}$ for every $l$, are well-defined morphisms and satisfy the duality axioms. Notice that although we do not assume that for every $l$, there be a unique subset $J_{l}\subset I$, we require a \emph{choice} of such a subset to enable us to define the evaluation morphism, $\ev$. 
\begin{lemma}\label{LSLRig} If $\mathcal{L}$ is a rigid object in $\SL$, then $\ct{L}$ has a basis. 
\end{lemma}
\begin{proof} Assume $\ct{M}$ is the right dual of $\ct{L}$ with duality morphisms $\ev : \ct{L}\tn \ct{M}\rightarrow \ct{P}(1)$ and $\cv : \ct{P}(1)\rightarrow \ct{M}\tn \ct{L} $. Let $\cv (1)= \lbrace (m_{i},l_{i})\mid i\in I\rbrace$. By the duality axioms e.g. $(\ev\tn\id_{\ct{L}})(\id_{\ct{L}}\tn \cv )=  \id_{\ct{L}} $, we can conclude that for any $l\in \ct{L}$, there exists a subset $J_{l}\subseteq I$ such that $\vee_{i\in J_{l}}l_{i} =l $ and $\ev \lbrace (l,m_{i})\rbrace = 1$ iff $ i\in J_{l}$. In particular, for a fixed $i\in I$, if $j\in J_{l_{i}} $, then $l_{j}\leq l_{i}$ and $(l_{i},m_{i})\geq (l_{j},m_{i})$. Hence, because $\ev$ is join-preserving, $\ev(l_{i},m_{i})= \ev(l_{j},m_{i}) =1$ for any $i\in I$. Additionally, $(\id_{\ct{M}}\tn \ev )(\cv\tn\id_{\ct{M}})=  \id_{\ct{M}} $, shows that $m_{i}\leq m_{j}$ if $\ev \lbrace (l_{j},m_{i})\rbrace = 1$.

By the symmetric structure of $\SL$, we can deduce that morphisms $\ev' : \ct{M}\tn \ct{L}\rightarrow \ct{P}(1)$ and $\cv' : \ct{P}(1)\rightarrow \ct{L}\tn \ct{M} $, defined by $\cv'(1)= \lb{l_{i}}{m_{i}}$ and $\ev' \lb{m}{l}  = \ev \lb{l}{m}$ make $\ct{M}$ the left dual of $\ct{L}$. Hence, if $\ev' \lb{m_{i}}{l_{j}}=1$, the same arguments as above show that $m_{j}\leq m_{i}$ and $l_{i}\leq l_{j}$. Therefore, $\ev(l_{j},m_{i})=1$ iff $i=j$.

Observe that if for $i,j\in I$, if $l_{j}\leq l_{i}$, then $\lb{l_{j}}{m_{j}}\leq \lb{l_{i}}{m_{j}}$ and the same line of arguments as above follow demonstrating that $i=j$. Hence, elements $l_{i}$ form a basis for $\ct{L}$.
\end{proof}

Let us briefly reflect on the fact that YBE solutions on rigid objects of $\SL$ correspond to set-theoretical YBE solutions. Of course the strong monoidal functor $\ct{P}:\mathrm{Set}\rightarrow \SL$ lifts the set-theoretical YBE solutions to YBE solutions on free lattices. Furthermore, an invertible morphism between lattices with bases must send basis elements to basis elements. Hence, a YBE solution on a lattice $\ct{L}$ with a basis is fully determined by a YBE solution on its basis. 
 
\section{Hopf Algebras in $\SL$}\label{SHpfSL}
In this section, we introduce the remnant group of a Hopf algebra in $\SL$ and show that any co-quasitriangular Hopf algebra gives rise to a skew brace. 

As pointed out earlier, for a lattice $\ct{L}$, the lattice $\ct{L}\tn\ct{L}$ can be viewed as a quotient of $\ct{P}(\ct{L}\times \ct{L})$. Hence, any element of can be written as $\cup_{i\in I}\lb{l_{i}}{l'_{i}}$ for some elements $l_{i},l'_{i}\in\ct{L}$. Therefore, we will use the shorter notation $\vee_{i\in I}(l_{i},l'_{i})$, instead of $\cup_{i\in I}\lb{l_{i}}{l'_{i}}$, unless otherwise stated.
\subsection{Remnant of a Hopf Algebra}\label{SRem}
The category $\SL$ has a symmetric monoidal structure via $\tn$, and thereby one can formulate the notion of a Hopf algebra in $\SL$. 

Let $\ct{H}$ be a complete lattice in $\SL$. An \emph{algebra} structure on $\ct{H}$, consists of a join-preserving morphism $m:\ct{H}\tn\ct{H}\rightarrow \ct{H}$ which is associative i.e. $m(m\tn \id_{\ct{H}})= m(\id_{\ct{H}}\tn m)$ and an element denoted by $1\in \ct{H}$, such that $m((1,h))=h=m((h,1))$, for all elements $h\in \ct{H}$. We will denote the multiplication of pairs $(h,h')\in \ct{H}\tn\ct{H} $, by $m(h,h')$ or $h.h'$ instead of $m((h,h'))$. 
\begin{rmk}\label{RMult} We should point out that although the multiplication is determined by its value on pairs of the form $(h,h')\in \ct{H}\tn\ct{H}$, it does not correspond to a monoid structure on the underlying set $\ct{H}$, but a stronger structure depending on the partial ordering, since $\vee_{i\in I} h.h_{i}= h. (\vee_{i\in I} h_{i})$ and $\vee_{i\in I}h_{i}.h=(\vee_{i\in I}h_{i}).h$ must hold.
\end{rmk} 

A \emph{coalgebra} structure on $\ct{H}$ consists of join-preserving morphisms $\Delta : \ct{H} \rightarrow \ct{H} \tn \ct{H} $, $\epsilon:\ct{H}\rightarrow \ct{P}(1)$ satisfying 
\begin{align*}
(\Delta\tn \id_{\ct{H}}) \Delta = (\id_{\ct{H}}\tn \Delta) \Delta, \quad (\epsilon\tn\id_{\ct{H}})\Delta = \id_{\ct{H}} = (\id_{\ct{H}}\tn \epsilon)\Delta
\end{align*}
For $h\in \ct{H}$, $\Delta (h)$ consists of the image of a subset of $\ct{H}\times \ct{H}$ in $\ct{H}\tn \ct{H}$. Hence, we adapt \emph{Sweedler's notation} from ordinary Hopf algebras and write $\Delta(h)=\lb{h_{(1)}}{h_{(2)}}$, where $(h_{(1)},h_{(2)})$ represent all the pairs appearing in $\Delta (h)$ and in fact $\lb{h_{(1)}}{h_{(2)}}=\vee \lbrace (a,b)\mid (a,b)\in \Delta (h)\rbrace$. Consequently, we can write
$$h= \vee \lbrace h_{(1)}\mid \epsilon (h_{(2)})= 1\rbrace = \vee \lbrace h_{(2)}\mid \epsilon (h_{(1)})= 1\rbrace $$
A \emph{bialgebra} structure on $\ct{H}$ consists of an algebra and coalgebra structures as above with additional compatibility conditions $\Delta (1)=(1,1)$, $\epsilon (1)=1 $, $\epsilon (m(h,h'))= 1$ iff $\epsilon (h)=\epsilon (h')=1$ and 
\begin{equation}\label{Ebialg} \lb{h_{(1)}.h'_{(1)}}{h_{(2)}.h'_{(2)}}= \lb{(h.h')_{(1)}}{(h.h')_{(2)}}
\end{equation}
We call $\ct{H}$ a \emph{Hopf algebra}, if there exists an invertible join-preserving morphism $S:\ct{H}\rightarrow\ct{H} $, such that $\vee h_{(1)}.S(h_{(2)})= \vee S(h_{(1)}).h_{(2)}$ is equal to $1$ iff $\epsilon (h)=1$ and equal to $\emptyset$ otherwise. 

\begin{ex}\label{EHpfGrpAlg} Since the power-set functor $\Pl{-}: \mathrm{Set}\rightarrow \SL$ is strong monoidal, the image of any Hopf algebra in $\mathrm{Set}$ i.e. $\Pl{G}$ for any group $G$, will have an induced Hopf algebra structure. Here, we refer to $\Pl{G}$ as the \emph{group algebra} of $G$. The induced Hopf algebra structure is defined by $\st{g}.\st{h}=\st{gh}$, $S(\st{g})=\st{g^{-1}}$, $1=\st{e}$ and $\epsilon (\st{g})=1$, where $g,h\in G$ and $e\in G$ is the identity element. Since the singleton sets $\st{g}$, where $g\in G$, form a basis of $\Pl{G}$, it is sufficient to define the structure on the basis elements. 
\end{ex}

\begin{ex}\label{EHpfFunctAlg} As demonstrated in Lemma \ref{LSLRig}, free lattices are rigid objects and in fact self-dual. Hence, we can provide a dual Hopf algebra structure on $\Pl{G}$ for a group $G$. We will refer to this structure as the \emph{function algebra}, which as for ordinary Hopf algebras, is defined by $\st{g}.\st{h} =\st{g}$ iff $g=h$, $S(\st{g})=\st{g^{-1}}$, $\Delta(\st{g})= \st{(h,l)\mid h.l=g} $, $1=\st{g\mid g\in G}$ and $\epsilon (\st{g})=1$ iff $g=1$.
\end{ex}

\begin{ex}\label{EDistHpf} Any complete lattice $\ct{L}$ which satisfies the following distributivity conditions $\vee_{i\in I} (a_{i}\wedge b)= (\vee_{i\in I} a_{i})\wedge b $ and $\vee_{i\in I} (b\wedge a_{i})= b\wedge (\vee_{i\in I} a_{i}) $, for $a_{i}, b\in \ct{L}$, obtains a natural Hopf algebra structure defined by $1=\vee \ct{L}$, $m(a,b)=a\wedge b$, $S=\id_{\ct{L}}$, $\epsilon(a)=1 $ iff $a= \vee\ct{L}$ and $\Delta (a )= \lb{a}{1}\vee \lb{1}{a}$. 	
\end{ex}

Let $\ct{H}$ be a Hopf algebra in $\SL$. We first observe that $\epsilon^{-1}(\emptyset)$ is a complete sublattice of $\ct{L}$. Hence, we can consider the parallel pair $\emptyset,\mathrm{inc.}:\epsilon^{-1}(\emptyset) \rightrightarrows \ct{H} $, where $\emptyset$ denotes the morphisms which sends all elements to $\emptyset\in \ct{H}$ and $\mathrm{inc.}$ denotes the natural inclusion morphism. We denote the coequalizer of this pair by $\pi :\ct{H}\rightarrow \ct{Q}$. Recall from the description of colimits in Section \ref{SSL}, that as a set $\ct{Q}$ is the quotient of $\ct{H} $ by relation $  h=\vee\big(\lbrace h \rbrace \cup \epsilon^{-1}(\emptyset)\big) $, and the morphism $\pi$ is defined by $\pi ( h)= \ov{h}$. If we denote $D:=\vee\epsilon^{-1}(\emptyset )$, then $\vee\big(\lbrace h \rbrace \cup \epsilon^{-1}(\emptyset)\big) = h\vee D$. Since the closure operator on $\ct{H}$, defined by $h\mapsto h\vee D$ is join-preserving, it follows that $\pi$ admits a well-defined join-preserving section $\iota :\ct{Q}\rightarrow \ct{H}$, defined by $\iota (\ov{h})=h\vee D$.

\begin{lemma}\label{LQuotHpf} The quotient lattice, $\ct{Q}$ of a Hopf algebra $H$, as described above has an induced Hopf algebra structure, such that $\pi$ becomes a Hopf algebra morphism i.e. $\pi$ commutes with all structural morphisms. 
\end{lemma}
\begin{proof} We define the Hopf algebra structure on $\ct{Q}$ denoted by $(m',\ov{1},\Delta', \epsilon',S')$ via 
\begin{align*}
m'=\pi m (\iota\tn\iota), \quad \Delta'= (\pi\tn\pi)\Delta\iota, \quad \epsilon'=\epsilon \iota, \quad S'=\pi S\iota
\end{align*}
Since $\iota\pi=\id_{\ct{Q}}$, then we only need to show that the defined morphisms induce a Hopf algebra structure. Checking that $m'(h,g)=\pi m(h\vee D, g\vee D) $ is associative, depends on showing that $\pi (h.(g\vee D))=\pi (h.g)= \pi ((h\vee D). g)$, for any $h,g\in H$. This follows from the fact that $h.D\in \epsilon^{-1}(\emptyset)$ and $\pi (h.(g\vee D))= \pi (h.g) \vee\pi ( h.D)  $. Consequently, $\pi ((1\vee D). (h\vee D))= \ov{h}= \pi ((h\vee D).(1\vee D))$ and $m',\ov{1}$ form an associative algebra structure on $\ct{Q}$.

Similarly, showing that $\Delta (h)= \lb{\pi (h_{(1)})}{\pi (h_{(2)})}$ is coassociative depends on showing that $(\pi\tn \pi)\Delta (h\vee D)= (\pi\tn \pi)\Delta (h) $. This fact follows from the observation that $\Delta (h\vee D)= \lb{h_{(1)}}{h_{(2)}}\vee \lb{D_{(1)}}{D_{(2)}} $ and $\lb{D_{(1)}}{D_{(2)}} \in \mathrm{Im}(\ct{H}\tn\epsilon^{-1}(\emptyset ) \cup \epsilon^{-1}(\emptyset)\tn\ct{H})\subset \ct{H}\tn\ct{H}$, since $\epsilon =(\epsilon \tn \epsilon )\Delta$. Note that $\epsilon'(\ov{h})= \epsilon (h)$ and it follows immediately that $\Delta'(\ov{h})= \lb{\ov{h_{(1)}}}{\ov{h_{(2)}}} $ defines a coalgebra structure on $\ct{Q}$.

The bialgebra axioms follow directly from our observations in the paragraph above, that $\pi (h.(g\vee D))=\pi (h.g)= \pi ((h\vee D). g)$ and $(\pi\tn \pi)\Delta (h\vee D)= (\pi\tn \pi)\Delta (h) $ and is left to the reader. Note that $S'(\ov{h})= \ov{S(h)}$, since $\epsilon S= \epsilon$ and thereby $S(D)\in \epsilon^{-1}(\emptyset )$. The Hopf algebra axioms then follow directly from this observation.
\end{proof}

As we noted in the proof of Lemma \ref{LQuotHpf}, $\epsilon'(\ov{h})=\epsilon (h)$ and thereby $\epsilon'(\ov{h})=\emptyset$ iff $\ov{h}=\emptyset$. We can further characterise such Hopf algebras. In what follows we will call an element $q\in \ct{Q}$ \emph{group-like}, if $\Delta (q)= \st{(q,q)}$. 
\begin{lemma}\label{LHpfGrp} Given a Hopf algebra structure on a complete lattice $\ct{Q}$, TFAE
\begin{enumerate}[(I)]
\item For any $\emptyset\neq q\in \ct{Q}$, $\epsilon ( q)= 1$. 
\item	 There exists a set $X$, such that $\ct{Q}\cong \Pl{X}$ as a lattice and a group structure on $X$ inducing a group algebra structure on $\ct{Q}$.
\end{enumerate}	
\end{lemma}
\begin{proof} $(I)\Rightarrow (II)$ We will prove this statement in several steps. 

Assume that for any $\emptyset\neq q\in \ct{Q}$, $\epsilon ( q)= 1$. First, note that for any $\emptyset\neq q\in \ct{Q}$, the coalgebra structure implies that $\vee \st{q_{(1)}}= q=\vee \st{q_{(2)}} $, so that for any pair $(q_{(1)}, q_{(2)} )\in \Delta (q)$, we have $q_{(1)}, q_{(2)}\leq q$. Secondly, we observe that $\vee \st{q_{(1)}.S(q_{(2)})}= 1$ and consequently $1\leq q.S(q) $.   

\textbf{\textit{Claim 1.}} If $\emptyset\neq q\leq 1$, then $q=1$.

Since $1\leq q.S(q)\leq q.S(1)\leq 1.S(1)=1 $ and $S(1)=1$, then $q=1$.

\textbf{\textit{Claim 2.}} For any $\emptyset\neq q\in\ct{Q}$, there exists at least one pair pair $(a,b)\in \lb{q_{(1)}}{q_{(2)}}$, such that $a.S(b)=1=S(b).a$. 

Since $\vee q_{(1)}.S(q_{(2)})= 1 $, then by Claim 1 it follows immediately that there exists a pair $(a,b)\in \lb{q_{(1)}}{q_{(2)}}$ such that $a.S(b)=1$. Consequently, by associativity we conclude that $p.a\neq \emptyset$ iff $p\neq \emptyset$. Since $(b,a)\in \lb{q_{(2)}}{q_{(1)}} $, thereby $S^{-1}(b).a\leq 1$. Hence, $S^{-1}(b).a= 1 $ and by associativity it follows that $S^{-1}(b)=S(b)$. A symmetric argument also shows that $S(a).b=1=b.S(a)$ and $S^{-1}(a)=S(a)$.

\textbf{\textit{Claim 3.}} If $a.b=1=b.a$, then $\lb{a_{(2)}}{a_{(1)}}=\lb{a}{a} $. 

The bialgebra axiom, \ref{Ebialg}, implies that $\lb {a_{(1)}.b_{(1)}}{a_{(2)}.b_{(2)}} =\lb{1}{1}$. Hence, for any $(a',b')\in \lb{a_{(1)}}{b_{(1)}}$, we have $a'.b' =1$ or $\emptyset$. On the other hand, Claim 2 shows that there exists a pair $(a',a'')\in \lb{a_{(1)}} {a_{(2)}} $ such that $S(a'').a'=1=S(a').a''$. Therefore, for any $(b',b'') \in \lb{b_{(1)}} {b_{(2)}} $, $(a'.b',a''.b'')=(1,1)$ and $\lb{b_{(1)}} {b_{(2)}} = \st{(S(a''),S(a'))}$. 
Since $b= \vee \st{b_{(1)}}=\vee \st{b_{(2)}}$, then $b= S(a')$ and $a'=a''$. By a symmetric argument, we see that $\Delta(a)=\lb{a'}{a'} $ and $a'=a=a''$.

So far we have shown that for every element $\emptyset\neq q\in \ct{Q}$, there exists a group-like element $a$, such that $(a,a) \in \lb{q_{(1)}}{q_{(2)}} $. Additionally, the multiplication on $\ct{Q}\tn\ct{Q}$ restricts to the set of group-like elements, $\mathrm{Grp}(\ct{Q})$, in $\ct{Q}$. 

\textbf{\textit{Claim 4.}} If $\emptyset\neq a\in \mathrm{Grp}(\ct{Q}) $ is group-like and $\emptyset\neq q\leq a$, then $q=a$.

The same argument as Claim 1, can be used $1\leq q.S(q)\leq q.S(a)\leq a.S(a)=1 $. By symmetric arguments, we conclude that $q$ is the inverse of $S(a)$, and by associativity $a=q$.

\textbf{\textit{Claim 5.}} If $\emptyset\neq q,p\in\ct{Q}$, and $q.p\neq \emptyset$.

By Claim 2, we see that there exist group-like elements $a,b\in\ct{Q}$ such that $a\leq q$ and $b\leq p$. Hence $\emptyset\lneq a.b \leq q.p$. 

Since for an arbitrary $\emptyset\neq q\in \ct{Q}$, and $(a,b)\in \lb{q_{(1)}}{q_{(2)}}$ $a.S(b)=1$ or $\emptyset$. By Claims 5 and 2, we conclude that $\Delta (q)= \vee_{i\in I} \lb{a_{i}}{a_{i}}$ for a set of group-like elements $a_{i}$ such that $q= \vee_{i\in I} a_{i}$. Furthermore, the set of group-like elements forms a basis for $\ct{Q}$. What remains to show is that the factorisation $q= \vee_{i\in I} a_{i}$ is unique. If $b$ is a group-like element such that $b\leq q$, then $\Delta (b)=(b,b)\leq \vee_{i\in I} \lb{a_{i}}{a_{i}} $. Since the group-like elements are minimal, by the lattice structure of $\ct{Q}\tn \ct{Q}= \Pl{\ct{Q}\times \ct{Q}}/ \sim$, we can conclude that the statement only holds if there exists an $i\in I$ such that $b=a_{i}$. Hence, $\ct{Q}\cong \Pl{\mathrm{Grp}(\ct{Q})}$.
 
$(II)\Rightarrow (I)$ is true by the definition of the Hopf algebra structure on the group algebra $\Pl{X}$.\end{proof}

\begin{corollary}\label{CoRem} Any Hopf algebra $\ct{H} $ in $\SL$, we has a corresponding group $R(\ct{H})$ such that the quotient Hopf algebra $\ct{H}/\epsilon^{-1}(\emptyset )$ is isomorphic to the induced group algebra of $R(\ct{H})$. We call $R(\ct{H})$ the \emph{remnant} of $\ct{H} $.   
\end{corollary}

\section{Co-quasitriangular Hopf algebras and Skew Braces}\label{SCQHASkw}
In this section, we show that the remnant of a co-quasitriangular Hopf algebra has an induced braiding operator. Consequently, we show that the secondary product $\star$ on a such a group agrees with the restriction of the transmuted product of the Hopf algebra to its remnant. 

Before we prove our results, we review the definitions of co-quasitriangular structures on Hopf algebras in $\SL$ and present Majid's transmutation theory in this case. Recall from the last section, that we are simply translating result on Hopf algebras in symmetric categories to the setting of $\SL$. But in doing so, our equations will look essentially the same as the ordinary Hopf algebra case, where instead of a implicit sum $a_{(1)}\tn a_{(2)}\in H\tn H$, we have an implicit supremum $\lb{a_{(1)}}{a_{(2)}}\in \ct{H}\tn\ct{H}$. In particular, for an arbitrary morphism $f:\ct{H}\tn\ct{H}\rightarrow \ct{L}$, by $f(a_{(1)}, a_{(2)}) $, we always mean $\vee\st{f(a_{(1)}, a_{(2)}) \mid (a_{(1)}, a_{(2)})\in \Delta (a)} $. 

We call a join-preserving morphism $\ct{R}:\ct{H}\tn\ct{H}\rightarrow \Pl{1}$ a \emph{co-quasitriangular} structure on a Hopf algebra $\ct{H}$, if the below conditions hold:
\begin{enumerate}[(A)]
\item $\ct{R}$	is convolution-invertible i.e., there exists a join-preserving morphism $\ct{R}^{-1} :\ct{H}\tn\ct{H}\rightarrow \Pl{1}$ satisfying 
\begin{equation}\label{EConvInv}
\Rr^{-1} (a_{(1)},b_{(1)}). \Rr (a_{(2)},b_{(2)}) =\epsilon (a).\epsilon (b)= \Rr (a_{(1)},b_{(1)})  . \Rr^{-1}(a_{(2)},b_{(2)}) \end{equation}
for any pair of elements $a,b\in \ct{H}$, where $.$ here signifies the natural monoid structure on $\Pl{1}$, defined by $1.0=0=0.1$, $1.1=1$ and $0.0=0$. 
\item For any $a,b,c\in\ct{H}$, the following equalities hold: 
\begin{align}
\Rr (a.b, c)&= \Rr(b,c_{(1)}).\Rr (a,c_{(2)})\label{ECQ1}
\\ \Rr (a,b.c)&= \Rr (a_{(1)},b).\Rr (a_{(2)}, c)\label{ECQ2}
\\ \Rr (b_{(1)}, a_{(1)})a_{(2)}&.b_{(2)}= b_{(1)}. a_{(1)} \Rr (b_{(2)} , a_{(2)})\label{ECQ3}
\end{align}
where in \ref{ECQ3}, we are using the natural action of $\Pl{1}$ on all lattices $\Pl{1}\tn \ct{L}\cong \ct{L}$, defined by $0.l=0$ and $1.l=l$.
\end{enumerate}

For any Hopf algebra, we can consider its category of \emph{left comodules}, denote by $\prescript{\ct{H}}{}{\SL}$, which has pairs $( \ct{L}, \delta: \ct{L}\rightarrow \ct{H}\tn \ct{L})$, satisfying $(\epsilon\tn\id_{\ct{L}}) \delta= \id_{\ct{L}} $ and $(\Delta\tn \id_{\ct{L}}) \delta=(\id_{\ct{H}}\tn \delta) \delta$, as objects and join-preserving morphisms which commute with the \emph{coactions}, $\delta$, appropriately as its morphisms. If we utilise notation $\delta (l)=\lb{l_{(0)}}{l_{(1)}}$, for elements $l\in \ct{H}$, so that $l_{(0)}\in \ct{L} $ and $l_{(1)}\in \ct{L}$, then the monoidal structure of $\SL$, lifts to the category of $\ct{H}$-comodules and the coaction on $\ct{L}\tn\ct{M}$ for arbitrary comodules $\ct{L}$ and $\ct{M}$ is defined by
$$\lb{(l,m)_{(0)}}{(l,m)_{(1)}}:= \lb{l_{0}.m_{(0)}}{(l_{(1)},m_{(1)})}$$
Consequently, when given a co-quasitriangular structure on $\ct{H}$, $\prescript{H}{}{\SL}$ is provided with an induce braiding $\psi$, defined by 
\begin{equation}
\psi_{\ct{L},\ct{M}}(l,m)= \Rr (l_{(0)}, m_{(0)}). (m_{(1)},l_{(1)})
\end{equation}
for arbitrary comodules $\ct{L}$ and $\ct{M}$.

\begin{rmk}\label{RCQHA} The equations, \ref{ECQ1}, \ref{ECQ2}, \ref{ECQ3}, described in the above definition	 are mirrors to the usual definition of co-quasitriangular Hopf algebras, as presented in Chapter 2 of \cite{majid2000foundations}. Notice, that this is because the definition we are presenting is used to provide a braiding on the category of left $\ct{H}$-comodules, while the description in \cite{majid2000foundations} is used for right comodules. Nevertheless, they are equivalent, since $\Rr^{-1}$ in our case will satisfy the axioms presented in \cite{majid2000foundations}.
\end{rmk}

It is known that any co-quasitriangular Hopf algebra has a well-defined braiding operator $\Gamma:\ct{H}\tn\ct{H}\rightarrow \ct{H}\tn\ct{H} $, defined by 
\begin{equation}\label{EBrdCQHA}
(a, b)\mapsto \Rr ( a_{(1)} , b_{(1)}). (b_{(2)},a_{(2)}) .\Rr^{-1} ( a_{(3)} , b_{(3)})
\end{equation}  	
for elements $a,b\in \ct{H}$, which make $(\ct{H},\Gamma)$ a braided object in $\SL$. Additionally, it follows from \ref{ECQ3}, that $\ct{H} $ is \emph{braided-commutative} with respect to this braiding i.e. $m\Gamma=m $ holds. It follows directly from \ref{ECQ1} and \ref{ECQ2} that 
\begin{align} 
\Gamma (a.b, c) = (\id_{\ct{H}}\tn m)&(\Gamma\tn \id_{\ct{H}})(a, \Gamma(b,c))\label{EGamma1}
\\\Gamma (a,b.c)=(m\tn \id_{\ct{H}}) &(\id_{\ct{H}}\tn \Gamma)(\Gamma(a,b),c)\label{EGamma2}
\end{align}
hold for $a,b,c\in \ct{H}$ and it should also be clear that $\Gamma(1,a)= (a,1)$ and $\Gamma(a,1)= (1,a)$. Consequently, we can show that $\Gamma$ restricts to a braiding operator on $R(\ct{H})$.

\textbf{Notation.} Since $R(\ct{H})$ is a subset of the quotient Hopf algebra, $\ct{Q}$, constructed in Lemma \ref{LQuotHpf}, we denote its elements by $\ov{a}$ for some $a\in \ct{H}$, where $\iota (\ov{a})=a\vee D$, with $D=\vee \epsilon^{-1}(\emptyset)$.
\begin{thm}\label{TRemCQHA} If $\ct{H} $ is a Hopf algebra in $\SL$ and $\Rr$ is a co-quasitriangular structure on $\ct{H} $, then the induced braiding, $\Gamma$, on $\ct{H} $ restricts to a braiding operator $r$ on the group $R(\ct{H})$.
\end{thm}
\begin{proof} Let $D$, $\ct{Q}$, $\pi$ and $\iota$ be as in Section \ref{SRem} and define $\xi: \ct{Q}\tn\ct{Q}\rightarrow \ct{Q}\tn\ct{Q}$ by $\xi= (\pi\tn\pi) \Gamma (\iota \tn \iota)$. Observe that by \ref{EConvInv}, we can conclude that $(\epsilon\tn \epsilon)\Gamma =\epsilon\tn \epsilon $. Consequently, for any $a,b\in \ct{H}$, the equality 
\begin{equation}\label{EprfGamma}
(\pi\tn\pi ) \Gamma (a, b\vee D) = (\pi\tn \pi ) \Gamma (a,b )	= (\pi\tn\pi ) \Gamma (a, b\vee D)
\end{equation}
holds. A symmetric argument proves the exact same statements where $\Gamma$ is replaced by $\Gamma^{-1}$. Hence, $(\pi \tn \pi )\Gamma^{-1} (\iota\tn \iota )(\pi\tn \pi) \Gamma (\iota \tn \iota)=  (\pi \tn \pi )\Gamma^{-1} \Gamma(\iota\tn \iota )= \id_{\ct{Q}\tn\ct{Q}}$. Thereby, $\xi$ is invertible and since  $\ct{Q}=\Pl{R(\ct{H})}$  as a lattice, by Lemma \ref{LHpfGrp}, $\xi $ must send basis elements to basis elements, and restricts to a bijective map $r:R(\ct{H}) \times R(\ct{H}) \rightarrow R(\ct{H}) \times R(\ct{H}) $, where $r(g,h)=\xi(\st{g},\st{h}) $.  

By the properties of $\Gamma$ described before (\ref{EGamma1}, \ref{EGamma2}), \ref{EprfGamma} and $\pi (h.(g\vee D))=\pi (h.g)= \pi ((h\vee D). g)$ from the proof of Lemma \ref{LQuotHpf}, we can immediately conclude that $\xi$ satisfies the same properties and $r$ is a braiding operator on $R(\ct{H})$. We will prove \ref{EbrdOpr1} as an example and leave the other properties to the reader. Let $\ov{a}, \ov{ b}, \ov{ c}\in \ct{Q}$, then
\begin{align*}
( \id_{\ct{Q}}\tn m') &(\xi\tn \id_{\ct{Q}})(\ov{ a},\xi(\ov{ b}, \ov{ c}))=(\id_{\ct{Q}}\tn\pi m (\iota\tn\iota))((\pi\tn\pi) \Gamma)\tn \id_{\ct{Q}} 	)
\\&(\id_{\ct{H}}\tn \iota\pi\tn \pi )(\id_{\ct{H}}\tn \Gamma)( a\vee D , b\vee D,c\vee D)
\\= & (\id_{\ct{Q}}\tn\pi m (\iota\tn\iota))(\pi\tn\pi\tn\pi)(\Gamma\tn \id_{\ct{H}})( a\vee D , \Gamma( b\vee D,c\vee D))
\\=& (\pi\tn\pi m)(\Gamma\tn \id_{\ct{H}})( a\vee D , \Gamma( b\vee D,c\vee D))
\\ =& (\pi\tn\pi)\Gamma( m (a\vee D,b\vee D) , c\vee D)= \xi (m'(\ov{ a}, \ov{ b}), \ov{ c}) \qedhere
\end{align*}
\end{proof}

Hence, the remnant of every co-quasitriangular Hopf alebra in $\SL$, becomes a group with a braiding operator. As described in Theorem \ref{TSkwMain}, any such group carries a secondary group structure, denoted by $\star$, which makes it a skew brace. We now show that the source of this second multiplication is \emph{transmutation}, \cite{majid1993transmutation}.

We recall the theory of transmutation of co-quasitriangular Hopf alebras from Chapter 9 of \cite{majid2000foundations}. Any Hopf algebra $\ct{H}$ can be viewed as an object of its category of left comodules via the \emph{left coadjoint coaction} defined by $ a\mapsto  \lb{a_{(1)}.S(a_{(3)})}{a_{(2)}}
$. In fact, $(\ct{H}, \Delta ,\epsilon)$ becomes a comonoid in $\prescript{\ct{H}}{}{\SL}$ i.e. $\Delta $ and $\epsilon$ commute with the appropriate coactions. If additionally $\ct{H}$ is equipped with a co-quasitriangular structure, one can defined a secondary multiplication and antipode, denoted by $\star$ and $S^{\star}$, respectively.     
\begin{align}
a\star b=& \Rr \left( S(a_{(2)}) \tn b_{(1)}S(b_{(3)})\right) a_{(1)}.b_{(2)} \label{ETrsm}
\\S^{\star}(a)=& \Rr \left(a_{(1)}\tn S(a_{(4)}) S^{2}(a_{(2)})\right) S(a_{(3)}) \label{ETrsmS} 
\end{align}  	
In this case, $(\ct{H}, \star, 1, S^{\star} , \Delta, \epsilon ) $ becomes a \emph{braided Hopf algebra} in the braided monoidal category $\prescript{\ct{H}}{}{\SL}$, Example 9.4.10 \cite{majid2000foundations}. A braided Hopf algebra in this case only differs from a Hopf algebra, in the statement of the bialgebra condition, \ref{Ebialg}, where the braiding of the category comes, into play, but we will not utilise this in what follows. 

\begin{thm}\label{TRemSkwB} Let $(\ct{H},\Rr)$ and $r$ be as in Theorem \ref{TRemCQHA}. The transmuted product on $\ct{H} $, \ref{ETrsm}, restricts to a secondary group structure on $X$ and agrees with the induced $\star$ multiplication of $(R(\ct{H}),r)$, from Theorem \ref{TSkwMain}.
\end{thm}
\begin{proof} First, we must consider the induced braiding operator $r$ more carefully. Let $\ov{a},\ov{b}\in \ct{Q}$, since $(\pi\tn\pi ) \Gamma (a\vee D, b\vee D)= (\pi\tn\pi ) \Gamma (a, b)$ then  
\begin{align*}
r(\ov{a},\ov{b})= \Rr ( a_{(1)} , b_{(1)}). \big(\pi(b_{(2)}),\pi(a_{(2)})\big) .\Rr^{-1} ( a_{(3)} , b_{(3)})
\end{align*}
We again adapt the notation $r(\ov{a}, \ov{b})= \big(\sigma_{\ov{a}}(\ov{b}),\gamma_{\ov{b}}(\ov{a})\big)$ of Section \ref{SSkwBasic}. It is a straightforward consequence of the definition of a a co-quasitriangular structure $\Rr$, that it must satisfy $\Rr (a,S(b))=\Rr^{-1}(a,b)$ for any $a,b\in\ct{H}$ [Lemma 2.2.2 \cite{majid2000foundations}]. Since $\epsilon'\pi=\epsilon$, and $\epsilon'$ is non-trivial on all elements other than $\emptyset$, then 
\begin{align}
\st{\sigma_{\ov{a}}(\ov{b})}&= (\id_{\ct{Q}}\tn\epsilon')\xi (\st{\ov{a}}, \xi\st{\ov{b}})\nonumber
\\&= 	\Rr ( a_{(1)} , b_{(1)}) .\Rr^{-1} ( a_{(3)} , b_{(3)}). \pi(b_{(2)}).\epsilon(a_{(2)})\label{Eprflambda}
\\ & = \Rr ( a_{(1)} , b_{(1)}) .\Rr ( a_{(2)} ,S( b_{(3)})). \pi(b_{(2)}) =\Rr \left( a \tn b_{(1)}S(b_{(3)})\right) \pi(b_{(2)})\nonumber
	\end{align}
Since, $\epsilon \tn\epsilon = \epsilon \star $, we can proceed as in the proof of Lemma \ref{LQuotHpf}, and observe that $\star$, restricts to an associative product,  $\star'$, on $\ct{Q}$. More importantly, $\pi (a\vee D \star b\vee D )= \pi (a\star b) $. Additionally, we note that for any $\ov{a}\in R(\ct{H})$ and for any $q\in \st{q\mid \exists b\in \ct{H}, (b,q)\in \Delta(a) }$, the value of $\pi (q )$ is either $\emptyset$ or $a$. Thereby, $( \pi\tn \id_{\ct{H}}) \Delta(a) = \vee_{\st{b\mid (q,b)\in\Delta (a),\ \epsilon (q)=1 }}(\ov{a},b)=(\ov{a},a)\in \ct{Q}\tn\ct{H}$. Hence,  
\begin{align*}
\ov{a}\star \ov{b}=& \Rr \left( S(a_{(2)}) \tn b_{(1)}S(b_{(3)})\right) .\pi(a_{(1)}).\pi(b_{(2)})
\\=& 	\Rr \left( S(a) \tn b_{(1)}S(b_{(3)})\right) .\ov{a}.\pi(b_{(2)})
\\ =& 	\Rr \left( S(a) \tn b_{(1)}S(b_{(3)})\right) .\ov{a}.\pi(b_{(2)})= \ov{a}.\sigma_{\ov{S(a)}}(\ov{b})= \ov{a}. \sigma_{\ov{a}^{-1}}(\ov{b})\qedhere
\end{align*}
\end{proof}

\begin{rmk}\label{RRightBr} As mentioned in Section \ref{SSkwBasic}, there is also a notion of \emph{skew right braces}, which we do not discuss here. Although, we do not present the details, the author believes the secondary group structure for skew right braces, should arise by applying the right-handed transmutation of the Hopf algebra i.e. the multiplication which makes $(\ct{H},\Rr)$ a braided Hopf algebra in the category of \emph{right} $\ct{H}$-comodules.    
\end{rmk}

As mentioned in \ref{SSkwBasic}, to any skew brace $(B,.,\star)$, one can associate a crossed group $(B,\star)\rtimes (B,.)$. On the other hand, in the theory of co-quasitriangular Hopf algebras, there is a well-known construction for new Hopf algebras called the cross product or \emph{bosonisation}, using the original Hopf algebra and a braided Hopf algebra in its braided category of comodules \cite{majid1994cross}. In particular, if we denote the transmutation of $\ct{H}$ by $\ct{H}_{\mathrm{ad}}$, we can form the cross product $\ct{H}_{\mathrm{ad}}\rtimes \ct{H}$ Hopf algebra on the underlying object $\ct{H}_{\mathrm{ad}}\tn \ct{H}$ [Corollary 4.6 \cite{majid1994cross}]. For the general theory of bosonisation as we will use here, we refer to Theorem 9.4.12 of \cite{majid2000foundations} and the description of the dual statement afterwards.   

\begin{corollary}\label{CCrssGrp} If $(B,.,\star)$ is the skew brace arising as the remnant of a co-quasitriangular Hopf alebra $(\ct{H},\Rr)$, then the crossed group $(B,\star)\rtimes (B,.)$ is the remnant of the Hopf algebra $\ct{H}_{\mathrm{ad}}\rtimes \ct{H}$.	
\end{corollary}
\begin{proof} First, we note that the induced Hopf algebra $\ct{H}_{\mathrm{ad}}\rtimes \ct{H}$ is an induced structure on the object $\ct{H}_{\mathrm{ad}}\tn \ct{H} = \ct{H}\tn \ct{H} $ and its counit takes the form $\epsilon (a,b)=\epsilon (a).\epsilon(b)$ 
for $a,b\in \ct{H}$ [Proposition 1.6.18 \cite{majid2000foundations}]. It follows directly, that the quotient Hopf algebra of $\ct{H}_{\mathrm{ad}}\rtimes \ct{H}$, in the sense of Lemma \ref{LQuotHpf}, must be of the form $\Pl{R(\ct{H})}\tn \Pl{R(\ct{H})} = \Pl{R(\ct{H})\times R(\ct{H})} $.

Hence, we know that as a set $R(\ct{H}_{\mathrm{ad}}\rtimes \ct{H})= B\times B$. Now, we only need to recall the multiplication on $\ct{H}_{\mathrm{ad}}\rtimes \ct{H} $ from Equation (9.51) of \cite{majid2000foundations}. Recall that $\ct{H}_{\mathrm{ad}}$ is an object of $\prescript{\ct{H}}{}{\SL}$, via its left coadjoint action. Hence, for $a,b,c,d\in \ct{H}$, we have the following induced multiplication
\begin{equation*}
(a,b). (c,d)= \Rr(b_{(1)}, c_{(1)}.S(c_{(3)})  ). (a.c_{(2)}, b_{(2)}.d)\end{equation*}	
As in the proof of Theorem \ref{TRemCQHA}, we observe that the induced multiplication on $B\times B$ for elements $\ov{a}, \ov{b}, \ov{ c}, \ov{ d}\in B$, will have the form 
\begin{equation*}
(\ov{ a}, \ov{ b}). (\ov{ c}, \ov{ d})= \Rr(b_{(1)}, c_{(1)}.S(c_{(3)})  ). \big(\pi (a ).\pi( c_{(2)}), \pi(b_{(2)}).\pi(d)\big)\end{equation*}
And from the observations in Theorem \ref{TRemCQHA} and in particular equation \ref{Eprflambda}, it follows that the multiplication agrees with the multiplication of the crossed group, \ref{ECrssGrp}.  
\end{proof}

\subsection{Free Hopf Algebras}\label{SFrHpfSL}
Hopf algebras in $\mathrm{Rel}$ were classified in [LYZ1] and quasitriangular structures on them were described in [LYZ2]. Consequently, we can classify co-quasitriangular Hopf alebra structures on free lattices in $\SL$. Here we briefly review the proofs of these results and our interpretation of these results in terms of the remnant. 

If $G$ is a group and $ G_{+},G_{-}$ are subgroups of $G$ such that for any element $g\in G$, there exists a unique pair $g_{+}\in G_{+} $ and $g_{-}\in G_{-}$ satisfying $g=g_{+}.g_{-}$ , $G$ is said to have a \emph{unique factorisation} denoted by $G=G_{+}.G_{-}$. By applying inverses, we see that $G$ also factorises as $G=G_{-}.G_{+}$, with notation $g=\ov{g}_{-}.\ov{g}_{+}$. Hence, one can define the following actions of $G_{+}$ and $G_{-}$ on each other:
\begin{align*}
g_{+}^{g_{-}}=\ov{g}_{+}, \quad \prescript{g_{+}}{}{g_{-}}= \ov{g}_{-}, \quad  \ov{g}_{-}^{\ov{g}_{+}}=g_{-}	
, \quad \prescript{\ov{g}_{-}}{}{\ov{g}_{+}}= g_{+}\end{align*}
For more details about the properties which these actions satisfy, we refer the reader to \cite{LYZ2,LYZ1}. For a group $G$ with a unique factorisation, $\Pl{G}$ admits a Hopf algebra structure defined by 
\begin{align*}
g.h=& \begin{cases} g.h_{-}= \ov{g}_{-}.h= \ov{g}_{-}.h_{+}h_{-} &\text{iff}\ h_{+}=\ov{g}_{+}
\\ \emptyset	 & \text{otherwise}
\end{cases} 
\\ 1=& \vee\st{g_{+}\mid g_{+}\in G_{+}}, \hspace{1cm} S(g)=g^{-1}
\\ \Delta(g)=& \vee_{h_{+}\in G_{+}} \lb{g_{+}h_{+}^{-1}(\prescript{h_{+}}{}{g_{-}})}{h_{+}g_{-}}
\\ \epsilon (g)=& \begin{cases}1 &\text{iff}\ g\in G_{-}
\\ \emptyset &\text{otherwise}
 \end{cases}
\end{align*}   
We should note that this structure arises as the \emph{bicrossproduct} [Example 6.2.11 \cite{majid2000foundations}] of the group algebra $\Pl{G_{-}}$ and function algebra $\Pl{G_{+}}$, from Examples \ref{EHpfGrpAlg} and \ref{EHpfFunctAlg}, respectively. 
\begin{thm}\label{TFrHpf}\cite{LYZ2,LYZ1} For a set $G$, any Hopf algebra structure on the power-set $\Pl{G}$ corresponds to a group structure on $G$ with a unique factorisation $G=G_{+}.G_{-}$, with the resulting Hopf algebra structure described above on $\Pl{G}$. 
\end{thm}
The proof of this statement is presented in \cite{LYZ1}, for finite-dimensional Hopf algebras with positive basis and is said to follow for free modules over the Boolean algebra with 2 elements i.e. free lattices. The positive basis assumption is replacing the fact that every element in $\Pl{G}$, must be the join of a unique set of basis elements. However, in this case things are much simpler, since there are no scalars to worry about. It is not difficult to follow the proof and check that all arguments do hold as already mentioned in the article. The first step of this proof is to identify $G_{+}=\st{g\in G\mid g\in 1}$ and observe that $\Pl{G_{+}}$ is a commutative Hopf subalgebra of $\Pl{G}$. The finiteness condition comes in use when the authors apply the same argument to dual Hopf algebra of $\Pl{G}$. However, we know that any free lattice is a dualizable object and finiteness is no longer an issue for defining the dual Hopf algebra structure on $\Pl{G}$. Recall that the dual Hopf algebra structure on $\Pl{G}$, is defined by $(\Delta^{\vee},\epsilon^{\vee},m^{\vee},\eta^{\vee}, S^{\vee})$, where $\eta:\Pl{1}\rightarrow \Pl{G}$ is the morphism which sends $1$ to the designated unit element $1\in \Pl{G}$. In this way $G_{-}=\st{g\in G\mid \epsilon(g)=1}$ and $\Pl{G_{-}}$ is shown to be a cocommutative Hopf subalgebra of $\Pl{G}$. The authors also use the classification of finite-dimensional cocommutative Hopf algebras, to show that $\Pl{G_{-}}$ is a group algebra. While for us, this follows from the properties of the counit on $\Pl{G_{-}}$ and Lemma \ref{LHpfGrp}.

Notice that by definition, the quotient algebra, $\ct{Q}$, of Lemma \ref{LQuotHpf}, will be isomorphic to $\Pl{G_{-}}$. We must emphasise that in general the quotient constructed in Lemma \ref{LQuotHpf}, will not be isomorphic to the sublattice $\epsilon^{-1}(1)\cup \emptyset$. It is only in this simple case, that they agree, and even here, the map $\iota$ is different to the natural inclusion of $\Pl{G_{-}}$ as a subalgebra. Nevertheless, in the next proof, this difference will be ineffective, due to the application of $\pi$. 

In \cite{LYZ2}, positive quasitriangular
structure of such Hopf algebras were classified. Again this statement holds for free lattices. It should be clear that the dual of the $\Pl{G}$, will again be $\Pl{G}$, but the dual Hopf algebra structure will reverse the factorisation and use $G=G_{-}.G_{+}$. Using this technique, we can also clasify all possible co-quasitriangular structures on $\Pl{G}$. Co-quasitriangular structure on $\Pl{G_{+}.G_{-}}$ correspond to a pair of group morphisms $\eta,\xi : G_{-}\rightarrow G_{+}$ satisfying 
\begin{align*}
\prescript{v}{}{\xi (u)}= \xi \big( \prescript{\eta(v)}{}{u}\big)\quad \eta(v)^{u}= \eta \big(v^{\xi (u)}& \big)  \quad uv= \big(\prescript{\eta (u)}{}{v}\big)\big(u^{\xi (u)}\big)  
\\ \prescript{u}{}{x}\eta (u^{x})= \xi (u) x\quad &\prescript{u}{}{x}\xi (u^{x}) =\eta (u)x
\end{align*} 
for $u,v\in G_{-}$ and $x\in G_{+}$ [Proposition 1 and Theorem 1 \cite{LYZ2}]. When provided with such a pair, we can define $\Rr:\Pl{G\times G}\rightarrow\Pl{1}$ by 
\begin{equation*}
\Rr(g, h)= \begin{cases}
1 & \text{iff there exists a pair}\ u,v\in G_{-}\ \text{s.t}\ g=v\xi (u),\ h= u(\eta(v)^{u})^{-1}
\\\emptyset &\text{otherwise}	
\end{cases}	
\end{equation*}
\begin{thm}\label{TFrHpfRem} Let $\ct{H}=\Pl{G}$ be the resulting Hopf algebra of a group $G=G_{+}G_{-}$ with unique factorisation, and $\eta,\xi : G_{-}\rightarrow G_{+}$ provide a co-quasitriangular structure on $\ct{H} $. The remnant of $\ct{H}$ is isomorphic to the group $G_{-}$ and its induced braiding operator is defined by
\begin{equation}\label{EBrUnqFac}
(g_{-}	,h_{-})\longmapsto \left( \prescript{\eta (g_{-})}{}{h_{-}},g_{-}^{\xi (h_{-})}\right)
\end{equation}
for $g_{-}	,h_{-}\in G_{-}$.
\end{thm}
\begin{proof} We have already mentioned why the remnant of $\ct{H}$ is the group $G_{-}$. We must describe the induced braiding operator
\begin{align*}
r(\ov{a},\ov{b})= \Rr ( a_{(1)} , b_{(1)}). \big(\pi(b_{(2)}),\pi(a_{(2)})\big) .\Rr^{-1} ( a_{(3)} , b_{(3)})
\end{align*}
in this case, where $\ov{a},\ov{b}\in G_{-}$. For $a\in G_{-}\subset G$, 
$$\st{(a_{(1)},a_{(2)},a_{(3)})}= \st{\big(l_{+}^{-1}(\prescript{l_{+}}{}{a}),l_{+}k_{+}^{-1} \prescript{k_{+}}{}{a} , k_{+}a\big)\mid l_{+},k_{+}\in G_{+}}$$
Since $\pi(g_{+}g_{-})=g_{-} $ iff $g_{+}= e$ and $\emptyset$ otherwise, then 
\begin{align*}
r(a,b)= \vee_{l_{+},k_{+}\in G_{+}}\Rr \big( l_{+}^{-1}(\prescript{l_{+}}{}{a}) , k_{+}^{-1}(\prescript{k_{+}}{}{b})\big). \big(\prescript{k_{+}}{}{b}, \prescript{l_{+}}{}{a}\big) .\Rr^{-1} ( l_{+}a , k_{+}b)
\end{align*}
We note that $l_{+}^{-1}(\prescript{l_{+}}{}{a})= a \big(\prescript{a^{-1}}{}{ l_{+}^{-1} }\big)$, by equations (1) of \cite{LYZ2}. The first term $\Rr \big( l_{+}^{-1}(\prescript{l_{+}}{}{a}) , k_{+}^{-1}(\prescript{k_{+}}{}{b})\big)$ takes the value $1$ if and only if $\xi (b) = \prescript{a^{-1}}{}{ l_{+}^{-1} } $ and $\prescript{b^{-1}}{}{k_{+}^{-1}} =\big(\eta (a )^{b}\big)^{-1}$. The first equation is resolved by $l^{-1}_{+}= \prescript{a}{}{\xi(b)}$ and for the second equation we recall that $\prescript{b^{-1}}{}{k_{+}^{-1}} = (k_{+}^{b})^{-1}$, from equations (2) of \cite{LYZ2}. Hence, $k_{+}= \eta (a)$. It is easy to check that $\Rr^{-1} ( l_{+}a , k_{+}b) $ is also only non-trivial for the same values of $l_{+},k_{+}\in G_{+}$. By (1) of \cite{LYZ2}, we conclude that $\prescript{l_{+}}{}{a} = \prescript{\xi(b^{-1})^{a^{-1}}}{}{a}= a^{\xi (b)} $, thereby demonstrating that $r(a,b)=\big(\prescript{\eta(a)}{}{b},a^{\xi(b)} \big). $
\end{proof}
\begin{ex}\label{EXUnqFac} [LYZ2] Given a group with a unique factorisation $G=G_{+}G_{-}$, there is a natural braiding operator on $G$ itself, sometimes referred to as Weinstein and Xu's solution \cite{weinstein1992classical}. In Section 7 of [LYZ2], it was pointed out that the relevant braiding operator on $G$ appears from the Drinfeld double of $\Pl{G}$.
\end{ex}

\section{FRT Reconstruction}\label{SFRT}
In this section, we use the FRT reconstruction as formulated in Section \ref{SBasics}, first to recover the universal group of a set-theoretical YBE solution, Theorem \ref{TFRTRem}, and secondly, to obtain a co-quasitriangular Hopf alebra for every group with a braiding operator, whose remnant recovers the group and the operator, Theorem \ref{TSkwCQHA}. 
\subsection{Reconstruction for a Set-Theoretical Solution}\label{SFRTUnivGrp}
Let $(X,r)$ be set-theoretical YBE solution with the notation presented in Section \ref{SSkwBasic} i.e. $r(x,y)=(\sigma_{x}(y),\gamma_{y}(x))$ and $r^{-1}(x,y)= (\tau_{x}(y),\rho_{y}(x)) $. First let us take the solution into the category of sets and relations $\mathrm{Rel}$, via the natural faithful functor $\mathrm{inc.}:\mathrm{Set}\rightarrow \mathrm{Rel}$.

It should be easy to see that a bijective YBE solution, is a braided object in $\mathrm{Set}$. Additionally, our assumption for non-degeneracy of the solution becomes equivalent to $(X,r)$ being dualizable in $\mathrm{Rel}$: 

Any set $X$ is dualizable in $(\mathrm{Rel},\times, \un)$ with itself as its dual and the evaluation and coevaluation morphisms given by 
\begin{align}
\ev= &\st{((x,x);1)\mid \forall x\in X}\subset X\times X\times \un   \label{Eev}
\\ \cv=& \st{(1;(x,x))\mid \forall x\in X}\subset \un\times X\times X   
\end{align}
and the image of these relations under $\Pl{-}:\mathrm{Rel}\rightarrow \SL$, become exactly the duality morphisms for $\Pl{X}$. Notice that morphisms in $\mathrm{Rel}$ are invertible, if and only if they describe bijective maps between the sets. If we denote the set $X$ when regarded as its own dual by $X^{\vee}$, we observe that non-degeneracy is necessary since $r^{\flat}(x,y)=(\gamma_{x}^{-1}(y), \rho_{y}^{-1}(x))$ and $(r^{-1})^{\flat}(x,y)=(\rho_{x}^{-1}(y), \gamma_{y}^{-1}(x)) $.

Hence, with reference to Section \ref{SBasics}, we have a strict monoidal functor $\omega:\tilde{\ct{B}}\rightarrow \mathrm{Rel}$, which sends $\mathbf{x}$ to $X$, $\mathbf{y}$ to $X^{\vee}$ and $\kappa_{\mathbf{x},\mathbf{x}}$ to $r$. The image of $\tilde{\ct{B}}$ forms a rigid braided monoidal subcategory of $\mathrm{Rel}$. Notice that included in this subcategory are morphisms $\omega(\kappa_{\mathbf{x},\mathbf{y}})$, $\omega(\kappa_{\mathbf{y},\mathbf{y}})$ and $\omega(\kappa_{\mathbf{y},\mathbf{y}})$ which define the braidings between $X$ and $X^{\vee}$, $X^{\vee}$ and $X$ and $X^{\vee}$ with itself, respectively. Using the laws of a braided category e.g. $\omega(\kappa_{\mathbf{x},\mathbf{y}}) = \omega\big(( \id_{\mathbf{x}\tn\mathbf{y}}\tn \ev)( \id_{\mathbf{y}}\tn\kappa^{-1}_{\mathbf{x}\tn\mathbf{x}}\tn \id_{\mathbf{y}})( \cv\tn \id_{\mathbf{x}\tn\mathbf{y}})\big)$, we can calculate these morphisms directly:
\begin{align*}
\omega(\kappa_{\mathbf{x},\mathbf{y}}):X\times X^{\vee}\rightarrow X^{\vee}\times X;& \quad (x,y)\longmapsto (\rho_{x}^{-1}(y), \gamma^{-1}_{y}(x))
\\\omega(\kappa_{\mathbf{y},\mathbf{y}}) :X ^{\vee}\times X\rightarrow X\times X^{\vee} ;& \quad (x,y)\longmapsto (\sigma_{x}^{-1}(y), \tau^{-1}_{y}(x))
\\\omega(\kappa_{\mathbf{y},\mathbf{y}}) :X^{\vee}\times X^{\vee}\rightarrow X^{\vee}\times X^{\vee} ;& \quad (x,y)\longmapsto (\rho_{x}(y), \tau_{y}(x))
\end{align*}
Let $\und{w}$ denote a finite sequence of values from $\st{-,\vee}$, so that we can denote the object $X\times X^{\vee}\times X ^{\vee} $ by $X^{\und{w}}$ for the sequence $\und{w}= (-,\vee,\vee)$. We denote the set of such sequences by $\ct{W}$ and define the inverse of a sequence $\und{w}=(w_{1},\ldots w_{n})$ by the sequence $\und{w}^{-1}=(w_{n}^{-1},\ldots w_{1}^{-1}) $, where $-^{-1}=\vee$ and $\vee^{-1}=-$. Hence, by the braiding principle of $\kappa_{\mathbf{a},\mathbf{b}\tn \mathbf{c} }= (\id_{\mathbf{b}}\tn \kappa_{\mathbf{a}, \mathbf{c}})(\kappa_{\mathbf{a}, \mathbf{b}}\tn\id_{\mathbf{c}}) $, we can calculate the induced braidings between any pair of objects $X^{\und{v}} $ and $X^{\und{w}}$, in the image of $\omega$. We denote this braiding by $r_{\und{v},\und{w}}: X^{\und{v}}\times X^{\und{w}} \rightarrow X^{\und{w}} \times X^{\und{v}} $ and extend our notation for $r$, so that for a pair of words $\und{x}\in X^{\und{v}} $ and $\und{y}\in X^{\und{w}}$, we write $r_{\und{v},\und{w}} (\und{x},\und{y})= (\sigma_{\und{x}}(\und{y}), \gamma_{\und{y}}(\und{x}))$ and similarly for $r^{-1}$.  

Now we apply the power-set functor to the above constructions and denote the strong monoidal functor $\Pl{\omega}: \tilde{\ct{B}}\rightarrow \SL$, by $\omega$. As mentioned in Section \ref{SSL}, $\SL$ is cocomplete and we can construct the corresponding coend, \ref{Ecoend}, for FRT reconstruction on $\omega$. 

Observe that 
\begin{align*}
H_{\omega}&=\int^{\mathbf{a}\in \tilde{\ct{B}}}	\omega(\mathbf{a})\tn\omega( \mathbf{a})^{\vee}	= \int^{\und{w}\in \ct{W}} \Pl{X^{\und{w}}\times X^{\und{w}^{-1}}} 
\\&= \coprod_{\und{w}\in \ct{W}} \Pl{X^{\und{w}}\times X^{\und{w}^{-1}}}\big/ \st{\text{Relations}}
 \end{align*}
where we are using the fact that the braidings on $X^{\vee\vee}$ agree with $X$ and $X^{\und{w}^{-1}}= (X^{\und{w}})^{\vee}$. By the symmetric structure of the category we can reorganise the elements of $X^{\und{w}}\times X^{\und{w}^{-1}}= X^{(w_{1},\ldots, w_{n},w_{n}^{-1},\ldots,w_{1}^{-1})} $ into $X^{(w_{1},w_{1}^{-1},\ldots, w_{n},w_{n}^{-1})}$, as long as we recall this change when deriving the mentioned relations and Hopf algebra structure. In this way $H_{\omega}$ along with its induced multiplication takes a simpler form since
\begin{equation*}
	\coprod_{\und{w}\in \ct{W}}\ct{P}\left(X^{(w_{1},w_{1}^{-1},\ldots, w_{n},w_{n}^{-1})}\right)
=\ct{P}\big( F(X\times X^{\vee}\sqcup  X^{\vee}\times X )\big) \end{equation*} 
where $F(S) $, for a set $S$, denotes the free monoid on the set $S$. If we look carefully at the induced multiplication on $H_{\omega} $, we can see that it must agree with the induced multiplication of the free monoid, hence this form of the coend is more desirable to work with. For the set $S= X\times X^{\vee}\sqcup  X^{\vee}\times X $, we will denote, elements of $X\times X^{\vee} $ and $X^{\vee}\times X $ by $(x,y)_{1}$ and $(x,y)_{2}$, respectively.

The category $\tilde{\ct{B}}$ is generated by two types of morphisms, namely braidings $\kappa$ and the duality morphisms $\ev,\cv$. Hence, there are two types of relations which must be quotiented out from $\Pl{F(S)}$ to obtain $H_{\omega}$. 

First we resolve the relations coming from the braidings. Let $F_{1}$, be the quotient of the monoid $F(S)$ as above, by the two-sided ideal generated by the following relations
\begin{align*}
\kappa_{\mathbf{x},\mathbf{x}};& \quad (x,y)_{1}.(a,b)_{1}= (\sigma_{x}(a), \sigma_{y}(b) )_{1}.(\gamma_{a}(x), \gamma_{b}(y))_{1}
\\\kappa_{\mathbf{x},\mathbf{y}};& \quad (x,y)_{1}.(a,b)_{2}= (\rho_{x}^{-1}(a), \rho_{y}^{-1}(b))_{2} .(\gamma^{-1}_{a}(x), \gamma^{-1}_{b}(y))_{1}
\\ \kappa_{\mathbf{y},\mathbf{x}};& \quad (x,y)_{2}.(a,b)_{1}= (\sigma_{x}^{-1}(a), \sigma_{y}^{-1}(b))_{1}.( \tau^{-1}_{a}(x), \tau^{-1}_{b}(y))_{2}
\\ \kappa_{\mathbf{y},\mathbf{y}};& \quad (x,y)_{2}.(a,b)_{2}= (\rho_{x}(a), \rho_{y}(b))_{2}.(\tau_{a}(x), \tau_{b}(y))_{2}
\end{align*}
for all $x,y,a,b\in X$. Now we must resolve the relations which arise from the evaluation and coevaluation morphisms. This is the step which would not be possible if we intended to construct this coend in $\mathrm{Rel}$. We apply a further adjustment to $F_{1}$. Consider the two sided ideal generated by elements 
\begin{align}\label{Edeltdelmnts}
(x,a)_{1}.(y,a)_{2}, \quad (a,x)_{2}.(a,y)_{1} \end{align}
for all $x,y,a\in X$, such that $x\neq y$, and denote it by $\ct{J}$. Let $F_{2}=F_{1}\setminus \ct{J}$ and observe that it no longer has a monoid structure since, the multiplication of certain pairs of elements is not defined, while $\Pl{F_{2}}$ continues to carry an algebra structure in $\SL$, where the pairs whose multiplication is undefined multiply to $\emptyset$. 

The above relations come from the fact that the evaluation morphism, \ref{Eev}, sending non-equal pairs to $\emptyset$. The lattice $H_{\omega}$, will be the quotient of $\Pl{F_{1}}$ by the following relations
\begin{align}\label{EImportnt}
 \st{f.(x,a)_{1}.(x,a)_{2}.h\mid a\in X}=\st{f.h}= \st{f.(a,x)_{2}.(a,x)_{1}.h\mid a\in X} \end{align}
for all $x\in X$ and $f,h\in F_{2}$. Hence, for any word $\und{w}\in \ct{W}$, we have morphisms $\mu_{\und{w}}: \Pl{X^{\und{w}}\times X^{\und{w}^{-1}} } \rightarrow H_{\omega}$ defined by the composition of the natural inclusions $\varrho_{\und{w}}: \Pl{X^{\und{w}}\times X^{\und{w}^{-1}} }\rightarrow \Pl{F(S)} $ and the induced projection $\varsigma: \Pl{F(S)} \rightarrow H_{\omega}$. It is straightforward to see from the relations imposed that $H_{\omega}$ along with morphisms $\mu_{\und{w}}$, becomes the coend of the mentioned diagram.  

Observe that in the last step we are constructing a quotient of the lattice structure, as formulated in Section \ref{SSL}, the relation automatically implies that for any pair $x,a\in X$ and $y\in F_{2}$, $\st{(a,x)_{2}.(a,x)_{1}.y}\leq \st{y}$ in $H_{\omega}$. Hence, notice that this lattice is far from admitting a basis since we have infinitely ordered chains where the order is strict e.g for any pair $x,a\in X$ and $y\in F_{2}$, we have a chain
\begin{equation*}
	\dots<\st{(a,x)_{2}.(a,x)_{1}.(a,x)_{2}.(a,x)_{1}.y}<\st{(a,x)_{2}.(a,x)_{1}.y}< \st{y}
\end{equation*}
Hence, the image of no element in $H_{\omega} $ is minimal i.e for every element $\emptyset\neq h\in H_{\omega}$, there exists an element $\emptyset\neq g\in H_{\omega}$, such that $g< h $ and $g\neq h$. 

Now we can describe the induced Hopf algebra structure on $H_{\omega}$, by Theorem \ref{TFRTmain}. As mentioned earlier the multiplication, is exactly the image of the multiplication for $\Pl{F_{2}}$, where certain pairs of elements multiply to give $\emptyset$ due to the reduction in the structure of $F_{2}$. Additionally, the image of the unit of $F(S)$, denoted by $\st{1}$, acts as the unit of $H_{\omega}$. 

Let us denote arbitrary elements $\st{(x_{1},y_{1})_{i_{1}}.\ldots. (x_{n},y_{n})_{i_{n}}}\in H_{\omega}$ by $(\und{x},\und{y})_{\und{i}}$, where $x=(x_{1},\ldots x_{n})$, $y=(y_{1},\ldots y_{n})$ and $\und{i}=(i_{1},\ldots i_{n}) $. The counit in this case is defined by $\epsilon ((\und{x},\und{y})_{\und{i}} )=1 $ if and only if $x_{i_{j}}=y_{i _{j}}$ for all $1\leq j\leq n$. Notice that $\epsilon $ is well-defined since it is invariant under the imposed relations on $F(S)$, which define $H_{\omega}$. The coalgebra structure is defined by 
\begin{equation*}
\Delta ((\und{x},\und{y})_{\und{i}} )=\vee\st{\big((\und{x},\und{l} )_{\und{i}},(\und{l},\und{y})_{\und{i}} \big)\mid \forall\ l\in X^{n}} 	
\end{equation*} 
Moreover, $H_{\omega}$ admits an involutive antipode defined by $S((\und{x},\und{y})_{\und{i}})= (\und{y}^{f},\und{x}^{f})_{\und{i}^{-1}}$, where $\und{x}^{f}$ denotes the sequence $\und{x}$ being flipped i.e. $\und{x}^{f}= (x_{n},\ldots, x_{1})$ and $\und{i}^{-1}$ denotes the sequence being flipped as well as $1$ and $2$ being switched e.g. $(1,1,2)^{-1}= (1,2,2)$. 

As described in Section \ref{SBasics}, $H_{\omega}$ will have an induced co-quasitriangular structure, $\Rr: H_{\omega}\tn H_{\omega}\rightarrow \Pl{1}$ defined by 
\begin{equation}
\Rr ( (\und{x},\und{y})_{\und{i}}, (\und{a},\und{b})_{\und{j}})= 1 \text{ iff } \und{y}=\sigma_{\und{x}}(\und{a})\text{ and } \und{b}=\gamma_{\und{a}}(\und{x}) 	
\end{equation}
where $\sigma,\gamma$ denote the extensions of the braiding to arbitrary $X^{\und{w}}$ and $X^{\und{v}}$.     
\begin{thm}\label{TFRTRem} Given a set-theoretical YBE solution $(X,r)$, the remnant of the Hopf algebra $H_{\omega}$ recovers the universal group $G(X,r)$ of the solution along with its braiding operator.\end{thm}
\begin{proof} We must form the quotient $\ct{Q}$ of $H_{\omega}$, as done in Lemma \ref{LQuotHpf}. First observe that by the definition of the counit $\epsilon$, we can write $\ct{Q}$ as a quotient of $\Pl{F(X\sqcup X^{\vee})}$, where we view $(\und{x},\und{y})_{\und{i}}$ with $\und{x}=\und{y}$ as the word $\und{x}_{\und{i}}$ in $F(X\sqcup X^{\vee})$. Notice that thereby the deleted elements \ref{Edeltdelmnts}, are ineffective and would be sent to $\emptyset$ anyway in $\ct{Q}$. Furthermore, $\ct{Q}$ should be written as a quotient of $\Pl{F(X\sqcup X^{\vee})/\ct{I}} $, where $\ct{I}$ is the two sided ideal generated by the set of relations
\begin{align*}
	x.a= \sigma_{x}(a).\gamma_{a}(x),
\hspace{0.6cm}& \hspace{0.6cm} x.b= \rho_{x}^{-1}(b). \gamma^{-1}_{b}(x)
\\ b.x= \sigma_{b}^{-1}(x).\tau^{-1}_{x}(b),
\hspace{0.6cm}& \hspace{0.6cm} y.b=  \rho_{y}(b). \tau_{b}(y)
\end{align*}
where $x,a\in X$ and $b,y\in X^{\vee}$. Lastly, if $D=\vee\epsilon^{-1}(\emptyset)$, we note that by relation \ref{EImportnt}, for any $x\in X$ and $f,h\in F_{2}$, we have
\begin{equation}
 \st{f.(x,x)_{1}.(x,x)_{2}.h}\vee D=\st{f.h}\vee D= \st{f.(x,x)_{2}.(x,x)_{1}.h}\vee D \end{equation}
Hence, the remaining relations imposed on $\ct{Q}$, will be that $\st{x_{1}}$ and $\st{x_{2}}$, are multiplicative inverses for any $x\in X$. In other words, $\ct{Q}$ is isomorphic to $\Pl{G}$, where $G$ is the group obtained by quotienting the \emph{free group} generated by $X$, $F_{g}(X)$, by the mentioned braiding relations, where elements $x_{2}\in X^{\vee}$ are now written as $x^{-1}$. It is straightforward to see that the latter three braiding relations then follow from the first namely $ x.a= \sigma_{x}(a).\gamma_{a}(x) $ for $a,x\in X$, and the inverse laws. Hence, $R(H_{\omega})\cong G(X,r)$ as groups.  

It remains to show that the induced braiding on $R(H_{\omega})$ agrees with that of the universal group of $(X,r)$. Let $\und{g},\und{h}\in G(X,r) $, then we recall the structure of the induced braiding on $R(H)$ from Theorem \ref{TRemSkwB} and observe that 
\begin{align*}
r(\und{g}, \und{h})=& \Rr ( (\und{g}, \und{g})_{(1)} , (\und{h}, \und{h})_{(1)}). \big(\pi((\und{g}, \und{g})_{(2)}),\pi((\und{h}, \und{h})_{(2)})\big) 
\\&.\Rr^{-1} ( (\und{g}, \und{g})_{(3)} , (\und{h}, \und{h})_{(3)})
\\=& \vee_{\und{l},\und{k},\und{m},\und{n}} \Rr ( (\und{g}, \und{l}) , (\und{h}, \und{m})). \big(\pi((\und{l}, \und{k})),\pi((\und{m}, \und{n}))\big) .\Rr^{-1} ( (\und{k}, \und{g})_{(3)} , (\und{n}, \und{h}))
\\ =& \big( \pi((\sigma_{\und{g}}(\und{h}), \sigma_{\und{g}}(\und{h})),\pi((\gamma_{\und{h}}(\und{g}), \gamma_{\und{h}}(\und{g}))\big)= \big(\sigma_{\und{g}}(\und{h}) , \gamma_{\und{h}}(\und{g})\big)  
\end{align*} 
where if $\und{g}\in X^{\und{w}}$, then $\und{l},\und{k}$ are take values in all elements of $X^{\und{w}}$ and similarly for $\und{m},\und{n}$.  
\end{proof}
In the proof of Theorem \ref{TRemSkwB}, we used the structural properties of the remnant, to deduce the projection of the transmuted product on the remnant. In the case of $H_{\omega}$, one can directly compute the multiplication and antipode of the transmutation of $H_{\omega}$ by \ref{ETrsm} and \ref{ETrsmS}. We present these structures for the interested reader and omit their verification:
\begin{align*}
(\und{x}.\und{y})_{\und{i}}	\star (\und{a}.\und{b})_{\und{j}} &= \left(\und{x}.\sigma^{-1}_{\und{y}}(a), \tau^{-1}_{\und{b}}(\tau_{\und{a}}(\und{y})). \rho_{\tau^{-1}_{\und{a}}(y)}(d) \right)_{\und{i}.\und{j}}
\\ S^{\star}( (\und{x}.\und{y})_{\und{i}}) &= \left( \rho_{\und{x}}^{-1}(x)^{f}, \mathfrak{l}(\sigma^{-1}_{\und{y}}(\und{x}))^{f}\right)_{\und{i}^{-1}}
\end{align*}
where $\mathfrak{l}$ is the unique isomorphism induced on the set $X$ and its powers, which we discuss in Appendix \ref{App}. Observe that when $\und{x}=\und{y}$ and $\und{a}=\und{b} $, the multiplication takes exactly the required form of $\star$ on the skew brace.   
\subsection{From Skew Braces to Co-quasitriangular Hopf Algebras}\label{SSkwCQHA}
Let $G$ be a group with a braiding operator $r:G\times G \rightarrow G\times G$. In this section, we construct a co-quasitriangular Hopf alebra in $\SL$, which recovers $(G,r)$ as its remnant. To do this  we apply the FRT construction as in last section, while replacing $\tilde{\ct{B}}$, with $\tilde{\ct{B}} _{m}$, which is the smallest rigid braided monoidal category generated by a commutative monoid. 

We define the category $\tilde{\ct{B}}_{m}$, by adding two new generating morphisms to $\tilde{\ct{B}}$, namely $m:\mathbf{x}\tn \mathbf{x} \rightarrow \mathbf{x} $ and $u: \un\rightarrow \mathbf{x} $, and imposing additional relations 
\begin{align*}
m(u\tn \id_{\mathbf{x}}&)= \id_{\mathbf{x}}  = m(\id_{\mathbf{x}}\tn u), \hspace{1cm}  m( m\tn \id_{\mathbf{x}}) = m(\id_{\mathbf{x}}\tn m) 
\\ m\kappa_{\mathbf{x}, \mathbf{x}} = m&, \hspace{1cm} \kappa_{\mathbf{x}, \mathbf{x}}(u\tn \id_{\mathbf{x}})= \id_{\mathbf{x}}\tn u, \hspace{0.5cm} \kappa_{\mathbf{x}, \mathbf{x}} (\id_{\mathbf{x}}\tn u) = u\tn \id_{\mathbf{x}} \end{align*} 
which make $\mathbf{x} $ a commutative monoid in $\tilde{\ct{B}}_{m}$.

Given a group $G$ with a braiding operator $r$, we can define a functor $\omega_{m}: \tilde{\ct{B}} _{m} \rightarrow \SL$, as before by sending $\mathbf{x} $ to $\Pl{G}$, morphisms $m$, $u$ and $\kappa$ to $\Pl{.}$, $\Pl{u}$ and $r$, respectively. This of course defines a strong monoidal functor and since $\tilde{\ct{B}} _{m}$ is a rigid braided category, the Hopf algebra $H_{\omega_{m}}$ constructed from the functor $\omega_{m}$ will have an induced co-quasitriangular structure.
\begin{thm}\label{TSkwCQHA} Given a group $G$ with a braiding operator $r$, the remnant of the reconstructed co-quasitriangular Hopf alebra, $H_{\omega_{m}}$, described above, recovers $(G,r)$. 
\end{thm}
\begin{proof} We must again first construct the coend $H_{\omega_{m}}$. As all the previous morphisms appear, $H_{\omega_{m}}$ will of course be a quotient of $H_{\omega}$, where $\omega :\tilde{B}\rightarrow \SL$ is the relevant functor for the underlying set-theoretical YBE solution $(G,r)$. The additional relations arise from the presence of the morphisms $m$ and $u$ in $\tilde{\ct{B}} _{m}$. The first relation comes from the parallel pair 
$$ \mu_{\mathbf{x}}(\mathbf{e}\tn \id_{\Pl{1}}), \mu_{\un}(\id_{\Pl{1}}\tn \mathbf{e}^{\vee}):\Pl{X\times 1} \rightrightarrows H_{\omega} $$
and imposess $\st{1}= \st{(e,e)_{1}}$. A symmetric relation coming from $\mathbf{e}^{\vee}$ imposes that $\st{1}= \st{(e,e)_{2}}$. The second set of relations arise from the parallel pair 
$$ \mu_{\mathbf{x}}(\Pl{.}\tn \id_{\Pl{X^{\vee}}}), \mu_{\mathbf{x}\tn\mathbf{x}}(\id_{\Pl{X\times X}}\tn \Pl{.}^{\vee}):\Pl{X\times X\times X^{\vee}} \rightrightarrows H_{\omega_{m}} $$ 
Consequently, for any $a,b,c\in G$ and $f,h\in F_{2}$ (where $F_{2}$ is as in the proof of Theorem \ref{TFRTRem}), we have the following relation 
$$\st{f.(a.b,c)_{1}.h}= \st{f.(a,d )_{1}. (b,e)_{1}.h\mid \forall d,e\in G\ \text{satisfying}\ d.e=c} $$	
There is also a symmetric relation, where $(,)_{1} $ is replaced by $(,)_{2}$. From these relation it should be clear that once we construct the remnant of $H_{\omega_{m}}$, the last relation, will imply that the image of any word $\und{x}= x_{1}x_{2}\ldots x_{n}\in G^{n}$ in $F_{g}(G)$ is identified with the multiplication of the sequence in $G$. The additional braiding relations in $H_{\omega_{m}}$, are assumed to commute with the multiplication since $r$ is a braiding operator and hence, do not affect the computation. Consquently, $R(H_{\omega_{m}})\cong G$ and $r$ becomes the induced braiding on $G$, with exactly the same arguments as in Theorem \ref{TFRTRem}. 
\end{proof}
\appendix
\section{Appendix: An Induced Bijection}\label{App}
In this section, we investigate a bijection $\mathfrak{l}: X\rightarrow X$, which is induced for any set-theoretical YBE solution $(X,r)$. We will first define $\mathfrak{l}$ and show it is indeed bijective, by explicit computation, before commenting on its categorical origin at the end of the section.

For this section, we adapt a graphical notation, in the same vain as \cite{LYZ}. If $(a,b)=r(x,y)$ holds, for $x,y,a,b\in X$, we draw 

\adjustbox{scale=0.8,center}{\begin{tikzcd}
x 
& y \arrow[dl, no head] & &a \arrow[dr, no head ]& b \arrow[dl, no head, crossing over] \\ a 
& b \arrow[from=ul, no head , crossing over]& & x&y  \end{tikzcd}}\captionof{figure}{Graphical notation}
If we read these diagrams from up to down, considering the values in the top line as the entries and the bottom line as outputs, it is well-known that the Yang-Baxter equation is equivalent to the following diagrams having the same output:

\adjustbox{scale=0.8,center}{\begin{tikzcd}
x 
& y\arrow[dl,no head] & z \arrow[d,no head] & &x \arrow[d,no head]& y &z \arrow[dl,no head] 
\\ .\arrow[d,no head] & .\arrow[ul, no head, crossing over] &z \arrow[dl,no head] & 
& x&. \arrow[dl,no head] &. \arrow[ul, no head, crossing over]\arrow[d,no head]
\\ .&. \arrow[dl,no head] &c \arrow[ul, no head, crossing over] & 
& a\arrow[d,no head]& . \arrow[ul, no head, crossing over]&. \arrow[dl,no head]
\\ a& b \arrow[ul, no head, crossing over]& c\ar[u,no head]& & a& b&c \arrow[ul, no head, crossing over]\end{tikzcd}}\captionof{figure}{Yang-Baxter Equation}
where a straight line, denotes the identity map on $X$.
 
Since for any $x\in X$, $\gamma_{x}$ is bijective, there exists a unique $a\in X$ such that $\gamma_{x}(a) =x$. We can in fact show that for any $a\in X$, there exists a unique $x\in X$ satisfying $\gamma_{x}(a) =x$:
\begin{lemma}\label{LA31} Let $(X,r)$ be a set-theoretical YBE solution, with notation as before, for any $a\in X$, there exists a unique $x\in X$ such that $\gamma_{x}(a) =x$. 
\end{lemma}
\begin{proof} We first prove the existence of such an $x\in X$. Let $a\in X$, and pick $p,y\in X$, such that $\sigma_{a}(p)=a$ and $\rho_{a}(y)=a$, and observe that for some $l\in X$, the following diagrams 

\adjustbox{scale=0.8,center}{\begin{tikzcd}
\tau_{y}(a) & a\arrow[dl,no head] & p \arrow[d,no head] & 
& \tau_{y}(a) \arrow[d,no head]& a &p \arrow[dl,no head] 
\\ y\arrow[d,no head] & a\arrow[ul, no head, crossing over] &p \arrow[dl,no head] & 
& \tau_{y}(a)& a \arrow[dl,no head] &\gamma_{p}(a) \arrow[ul, no head, crossing over]\arrow[d,no head]
\\ y & a \arrow[dl,no head] & \gamma_{p}(a) \arrow[ul, no head, crossing over] & 
& y\arrow[d,no head]& a \arrow[ul, no head, crossing over]& \gamma_{p}(a) \arrow[dl,no head]
\\ ?& l \arrow[ul, no head, crossing over]& \gamma_{p}(a)\ar[u,no head]& 
& y& l&?? \arrow[ul, no head, crossing over]\end{tikzcd}}\captionof{figure}{Proof of existance}
hold. Hence, $?=y$ and $??= \gamma_{p}(a) $, so that $\gamma_{\gamma_{p}(a) }(a)= \gamma_{p}(a) $.

Now, we assume there exists a $b\in X$ such that $\gamma_{x}(b)=x $ and $\gamma_{y}(b)=y$ for $x\neq y$. Hence, since, $\sigma_{x}$ is bijective, pick $l$ so that $\sigma_{x}(l)=y$ and observe that for some $p,q,m,t\in X$, the following diagrams  

\adjustbox{scale=0.8,center}{\begin{tikzcd}
b & x\arrow[dl,no head] & l \arrow[d,no head] & &b \arrow[d,no head]& x &l \arrow[dl,no head] 
\\ p\arrow[d,no head] & x\arrow[ul, no head, crossing over] &l \arrow[dl,no head] & 
& b&y \arrow[dl,no head] &m \arrow[ul, no head, crossing over]\arrow[d,no head]
\\ p & y \arrow[dl,no head] &m \arrow[ul, no head, crossing over] & 
& q\arrow[d,no head]& y \arrow[ul, no head, crossing over]&m \arrow[dl,no head]
\\ q& t \arrow[ul, no head, crossing over]& m\ar[u,no head]& 
& q& t&? \arrow[ul, no head, crossing over]\end{tikzcd}}\captionof{figure}{Proof of uniqueness }
hold. Hence, $?=m $ and we have that $\gamma_{m}(y)=m$. Now we pick $j\in X$ such that $\sigma_{m}(j)=m$ and observe that for some $t\in X$, the figures

\adjustbox{scale=0.8,center}{\begin{tikzcd}
x & l\arrow[dl,no head] & j \arrow[d,no head] & &x \arrow[d,no head]& l &j \arrow[dl,no head] 
\\ y\arrow[d,no head] & m\arrow[ul, no head, crossing over] &j \arrow[dl,no head] & 
& x&? \arrow[dl,no head] &j \arrow[ul, no head, crossing over]\arrow[d,no head]
\\ y& m \arrow[dl,no head] & \gamma_{j}(m) \arrow[ul, no head, crossing over] & 
& t\arrow[d,no head]&m  \arrow[ul, no head, crossing over]&j \arrow[dl,no head]
\\ t& m \arrow[ul, no head, crossing over]& \gamma_{j}(m)\ar[u,no head]& 
& t& m&\gamma_{j}(m) \arrow[ul, no head, crossing over]\end{tikzcd}} \captionof{figure}{Proof of uniqueness }
hold. Hence $?=m$ and $y=\tau_{t}(m)=x $.
\end{proof}

In the situation of Lemma \ref{LA31}, if $\gamma_{x}(a) =x$, we will denote $\sigma_{a}(x)$ by $\mathfrak{l}(a)$. Observe that a completely symmetric argument shows that for any $a\in X$, there exists a unique $y\in X $ such that $\tau_{y}(a)=y$. In this case, we denote $\rho_{a}(y)$ by $\mathfrak{r}(a)$. From Figure 3, we can say that for any $a\in X$, there exist unique elements $x,y,l\in X$ such that $r(a,x)=(l,x)$ and $r (y,a)=(y,l)$. Hence, $\mathfrak{l}(a)=l$ and $\mathfrak{r}(l)=a$, making the maps $\mathfrak{l}$ and $\mathfrak{r}$ inverses, and consequently, bijections. In our diagrammatic notation, we have that the following diagrams hold.

\adjustbox{scale=0.8,center}{\begin{tikzcd}
a & x \arrow[dl, no head] & &y \arrow[dr, no head ]& \mathfrak{l}(a) \arrow[dl, no head, crossing over] 
\\ \mathfrak{l}(a) & x \arrow[from=ul, no head , crossing over]& & y& a \end{tikzcd}}
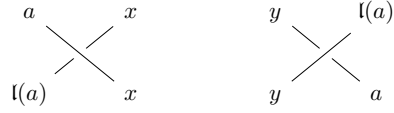
\captionof{figure}{Unique pair $a, \mathfrak{l}(a)$}

In fact, in the proof of Lemma \ref{LA31}, we have explicitly given $x,y,\mathfrak{l}(a)$ in terms of $a\in X$. In the paragraph before Figure 3, we chose $x=\tau_{a}^{-1}(a)$, $y=\rho_{a}^{-1}(a) $ and from Figure 3, we observe that $\mathfrak{l}(a)= \sigma_{a}(\tau_{a}^{-1}(a))= \gamma_{a}(\rho_{a}^{-1}(a))$.

As mentioned throughout this work, a set-theoretical YBE solution $(X,r)$ provides a functor $\omega:\tilde{\ct{B}}\rightarrow \mathrm{Rel}$. From this statement, it follows that the morphisms $\cv':=\omega(\kappa^{-1}_{\mathbf{y}, \mathbf{x}})\cv:\un \rightarrow X\times X^{\vee}$ and $\ev':=\ev \omega (\kappa_{\mathbf{y},\mathbf{x}}):X^{\vee}\times X\rightarrow 1$, must also satisfy the duality axioms. When written explicitly, the two morphisms, take the forms
\begin{align*}
\ev'= &\st{((x,\mathfrak{l}(a));1)\mid \forall x\in X}\subset X\times X\times \un   
\\ \cv'=& \st{(1;(x,\mathfrak{r}(x))\mid \forall x\in X}\subset \un\times X\times X  
\end{align*}
which as we have demonstrated above, are well-defined and satisfy the duality axioms.

\bibliographystyle{plain}

\end{document}